\documentclass[11pt]{article}
\usepackage{amsmath}
\usepackage{pifont}
\usepackage{amssymb}
\usepackage{amsfonts}
\usepackage{amscd}
\usepackage{epsfig}
\usepackage{amsmath,amstext,amsthm,amsbsy,amssymb}
\usepackage{graphicx}
\usepackage{epsfig}

\newcommand{\bc}{{\mathbb C}}
\newcommand{\br}{{\mathbb R}}
\newcommand{\bh}{{\mathbb H}}
\newcommand{\ba}{{\mathbb A}}

\newcommand{\bm}{{\mathbb M}}
\newcommand{\bp}{{\mathbb P}}
\newcommand{\F}{{\mathbb F}}

\newcommand{\bi}{{\bf i}}
\newcommand{\bj}{{\bf j}}

\newcommand{\p}{{\bf p}}
\newcommand{\q}{{\bf q}}
\newcommand{\e}{{\bf e}}
\newcommand{\z}{{\bf z}}
\newcommand{\bu}{{\bf u}}

\newcommand{\w}{{\bf w}}

\newcommand{\pspn}{{\rm PSp}(n,1)}
\newcommand{\pspt}{{\rm PSp}(2,1)}
\newcommand{\spn}{{\rm Sp}(n,1)}
\newcommand{\spt}{{\rm Sp}(2,1)}

\newcommand{\hq}{{\bf H}_{\bh}^n}

\newcommand{\bhq}{\partial {\bf H}_{\bh}^n}
\newcommand{\bhqt}{\partial {\bf H}_{\bh}^2}

\newcommand{\bhct}{\partial{\bf H}_{\bc}^2}

\newtheorem{thm}{Theorem}[section]
\newtheorem{lem}{Lemma}[section]
\newtheorem{prop}{Proposition}[section]

\newtheorem{dfn}{Definition}[section]
\newtheorem{rem}{Remark}[section]
\newtheorem{example}{Example}[section]

\newcommand{\fpm}{\fp=(p_1,\cdots,p_m)}

\newcommand{\fp}{{\mathfrak{p}}}

\begin{document}

\title{The moduli space of  points in  quaternionic projective space}
\author{ Wensheng Cao \\
School of Mathematics and Computational Science,\\
Wuyi University, Jiangmen, Guangdong 529020, P.R. China\\
e-mail: {\tt wenscao@aliyun.com}\\
}
\date{}
\maketitle

\bigskip
{\bf Abstract}\ \  Let $\mathcal{M}(n,m;\F \bp^n)$ be the configuration space of  $m$-tuples of pairwise distinct points in  $\F \bp^n$, that is, the quotient of the set of $m$-tuples of pairwise distinct points in  $\F \bp^n$  with respect to the diagonal action of  ${\rm PU}(1,n;\F)$ equipped with the quotient topology.  It is an important problem in hyperbolic geometry  to parameterize  $\mathcal{M}(n,m;\F \bp^n)$ and study  the geometric and topological  structures on the associated parameter space.
In this paper, by mainly using  the rotation-normalized  and block-normalized algorithms,  we construct  the parameter spaces   of   both $\mathcal{M}(n,m; \bhq)$ and $\mathcal{M}(n,m;\bp(V_+))$, respectively.

\bigskip



\section{Introduction}\label{se1-intr}
Let $\F=\br,\bc$ or $\bh$  be  respectively the real numbers, the complex numbers  or the quaternions, and
$\langle\z,\,\w\rangle=\w^*J\z$  a Hermitian product in $(n+1)$-dimensional $\F$-vector space  $\F^{n,1}$ of signature $(n,1)$, where $\z=(z_1,\cdots,z_{n+1})^T$,  $\w=(w_1,\cdots,w_{n+1})^T$  and   $\cdot^*$ denotes the conjugate transpose.  The group of transformations of $\F^{n+1}$ that preserve this Hermitian product is the noncompact Lie group $U(1,n;\F)$, that is, $$U(1,n;\F)=\{g\in {\rm GL}(n+1,\F):g^*Jg=J\}.$$
 These groups are traditionally  denoted by ${\rm O}(n,1)={\rm U}(1,n;\br)$, ${\rm U}(n,1)={\rm U}(1,n;\bc)$  and ${\rm Sp}(n,1)={\rm  U}(1,n;\bh)$.
Denote  by $\bp$  the  natural  right  projection from $\F^{n,1} -\{0\}$ to projective space $\F \bp^n$.   Let $V_-,V_0,V_+$ be the subsets of   $\F^{n,1} -\{0\}$  consisting
of vectors where  $\langle \z,\z\rangle$ is negative, zero, or positive, respectively.
Their projections to $\F \bp^n$ are called isotropic, negative, and positive points, respectively.  Conventionally, we  denote  ${\bf H}_\F^n=\bp(V_-),\partial{\bf H}_\F^n=\bp(V_0)$ and $\overline{{\bf H}_\F^n}={\bf H}_\F^n\cup \partial{\bf H}_\F^n$.
The Bergman metric on ${\bf H}_\bh^n$ is given by the distance formula
\begin{equation}\label{disformula}
\cosh^2\frac{\rho(z,w)}{2}=\frac{\langle\z,\,{\bf
		w}\rangle \langle\w,\,\z\rangle}{\langle\z,\,{\bf
		z}\rangle \langle\w,\,\w\rangle},\ \ \mbox{where}\ \ {\bf
	z}\in \bp^{-1}(z),\w\in \bp^{-1}(w).
\end{equation}
 The center ${\rm Z}(1,n;\F) $ in  ${\rm U}(1,n;\F)$  is  $\{\pm I_{n+1}\}$ if $\F=\br,\bh$, and is the circle group $\{e^{{\bf i}\theta} I_{n+1}\}$ if $\F=\bc$.  We mention that  $g\in U(1,n;\F)$  acts on  $\F \bp^n$ as $g(z)=\bp g \bp^{-1}(z)$. Therefore the holomorphic isometry group ${\rm Isom}({\bf H}_\F^n)$ of  ${\bf H}_\F^n$ is actually the quotient  ${\rm P U}(1,n;\F)={\rm  U}(1,n;\F)/{\rm Z}(1,n;\F)$.  We refer to \cite{apakim07,caogd16,chegre74,gol99,kimpar03} for further details.

 Let $\mathcal{M}(n,m;\F\bp^n)$ be the configuration space of  $m$-tuples of pairwise distinct points in  $\F \bp^n$, or equivalently, the quotient of the set of $m$-tuples of pairwise distinct points in  $\F \bp^n$  with respect to the diagonal action of  ${\rm PU}(1,n;\F)$ equipped with the quotient topology.
 It is an important problem in hyperbolic geometry  to parameterize the space $\mathcal{M}(n,m;\F \bp^n)$ and study the geometric and topological structures on the associated parameter space. We refer to such  a problem   the moduli problem on $\F \bp^n$.

The moduli problems  of the cases  $m = 1, 2$  on  $\partial{\bf H}_\F^n$ are trivial because  ${\rm U}(1,n;\F)$  acts doubly transitively  on $\partial{\bf H}_\F^n$ when $\F=\bc$ or $\bh$.   It is well-known that ${\rm O}(n,1)$ acts triply transitively  on  the boundary.  To handle the cases of  $m\geq 3$,  one need to develop  some geometric invariants or geometric tools, such as distance formula,  Cartan's angular  invariant \cite{car32,gol99}, and cross-ratio  \cite{korrei87} etc.

The moduli problem  of $\mathcal{M}(2,4; \bhct)$ was considered by Falbel, Parker and Platis \cite{falpla08,fal09,park08,parplal09}. The main tool  is the complex cross-ratio variety determined by three complex cross-ratios.

The moduli problem of  $\mathcal{M}(n,m; {\bf H}_\bc^n)$  was solved by Brehm and Et-Taoui \cite{bre90,bre98}.  Using   Bruhat decomposition,  Hakim and Sandler \cite{hasa00}  could construct  many important  geometric invariants in complex hyperbolic geometry. This  tool helped  them to arrange $n$  points in certain standard position on $\br \bp^{n-1}$ \cite{hasa0J}, and as well,  to  deal with   the moduli problem on $\overline{{\bf H}_\bc^n}$ \cite{hasa03}.

 We need to introduce the concept of Gram matrices  of  $m$-tuples in $\F \bp^n$ for further discussion.
\begin{dfn}\label{dfngram}
	Given an $m$-tuple $\mathfrak{p}=(p_1,\cdots, p_m)$ of pairwise distinct points in $\F \bp^n$ with lift  ${\bf p}=(\p_{1},\cdots, \p_{m})$.
	The following  Hermitian  matrix
	$$G(\p)=(g_{ij})=(\p_i^*J\p_j)=(\langle \p_{j},\p_{i}\rangle)$$
	is called the  {\it Gram matrix}  associated to $\mathfrak{p}$.
\end{dfn}
For the sake of simplicity, by a little abuse of notation,   we also say that  $\p$  is an $m$-tuple of pairwise distinct points in $\F \bp^n$  and  regard  $\p$  as an element in $\F_{n+1,m}$, the set of  $(n+1)\times m$ matrices over $\F$.
The  action of $f\in {\rm U}(1,n;\F)$ on $\F_{n+1,m}$  is the usual  matrix multiplication, that is,
$$f\p=(f\p_{1},\cdots, f\p_{m}).$$
Noting that $f^*Jf=J$,  we  have the following proposition.
\begin{prop}
	\begin{equation}\label{isomap}G(\p)=\p^*J\p=\p^*f^*Jf\p=G(f\p),\ \forall  f\in {\rm U}(1,n;\F).\end{equation}
\end{prop}

Given  two $m$-tuples $\mathfrak{p}=(p_1,\cdots, p_m)$  and $\mathfrak{q}=(q_1,\cdots, q_m)$  in $\F \bp^n$ with arbitrary lifts $\p=(\p_{1},\cdots, \p_{m})$ and   $\q=(\q_{1},\cdots, \q_{m})$.  We say that  $\mathfrak{p}$ and $\mathfrak{q}$  are ${\rm PU}(1,n;\F)$-congruent if there exists an $f\in {\rm U}(1,n;\F)$ such that   $$f(\p_i)=\q_i\lambda_i, \lambda_i\neq 0, i=1,\cdots,m,$$
 in language of matrix algebra,  that is,   $$f\p=\q D, \, D={\rm diag}(\lambda_1,\cdots,\lambda_m), \lambda_i\in \F- \{0\}.$$
Therefore
\begin{equation}\label{isofp1}G(\p)=\p^*J\p=\p^*f^*Jf\p=D^*\q^*J\q D=D^*G(\q)D.\end{equation}
Observe that  an arbitrary lift   of  $\mathfrak{p}$  can be represented by  $(\p_{1}\lambda_1,\cdots, \p_{m}\lambda_m)=\p D$ and
\begin{equation}\label{isolift1} G(\p D)=D^*\p^*J\p D=D^*G(\p)D.\end{equation}
The formulae (\ref{isofp1}) and (\ref{isolift1})  imply that   Gram matrices contain the information of  the diagonal action of ${\rm U}(1,n;\F)$  on $\p$.   Moreover, a  Gram matrix contains  the entries $\langle \p_{i},\p_{j}\rangle$, which  are  base material  to construct  the corresponding  Hermitian geometric invariants.  Hence  Gram matrix is the priority  tool in handling the moduli problem.

The moduli problem on  $\partial{\bf H}_\bc^n$ was solved by Cunha and Gusevskii \cite{cungus10,cungus12} mainly by Gram matrix.  The key idea  is that one need find a suitable matrix $D$ in (\ref{isolift1}) to construct corresponding normalized Gram matrix and then seek a bijection  between the independent  entries of  normalized Gram matrix  and  those geometric invariants   of  the parameter space  presenting  $\mathcal{M}(n,m;\partial{\bf H}_\bc^n)$.  We mention that the normalized processes  in \cite{cungus10,cungus12} and  the applications of Bruhat decomposition in \cite{hasa00}  share the some spirit  in eliminating the indeterminacy of $D$ in (\ref{isolift1}).

Let  $i(G(\p))=(n_+,n_-,n_0)$ be  the signature of Hermitian matrix $i(G(\p))$  and $V={\rm span}\{\p_1,\cdots, \p_m\}$ be of dimension $k+1$.
 There are two different cases of  the  moduli problem on  $\bp(V_+)$ according to $ n_++n_-=k+1$ or $ n_++n_-=k$ (see Theorem \ref{thm-inertia}).
   $V$ is called  parabolic in the latter case in \cite{chegre74}.  The two cases are termed by  regular and non regular cases in complex hyperbolic plane \cite{cungus12fr}. We still use this terminology in quaternionic setting.  In non regular case, the Gram matrices are unable to distinguish different  congruence classes.  In regular case, the orthogonality of positive points  always prevents one from taking similar normalized process in \cite{cungus12} and makes it extremely  difficult to find  the bi-directional recover process between   the geometric invariants and its corresponding  Gram matrix.   Cunha et al  surmounted  these difficulties  with exquisite techniques on complex hyperbolic plane \cite{cungus12fr}.

It is interesting to consider the moduli problems in  quaternionic hyperbolic geometry.
 However, one may encounter the difficulty  caused  by   the noncommutativity of quaternions.  Due to  this   noncommutativity, it is always  a huge challenge to do  computations in  quaternionic setting \cite{bisgen09,cpshimizu,kimpar03}.  Also, though in the literature  there have been  counterparts  of  terminologies such as  rank, determinant  and trace  which are extensively used in  commutative field,  the properties of  these concepts  may be much  different  in  quaternionic setting. One should be cautious to use them in noncommutative environment.  Furthermore,  another  essential difference between complex and quaternionic hyperbolic geometry is due to the  existence of elliptic elements of forms $\mu I_{n+1}$ in $\spn$, where $\mu\in {\rm Sp}(1)$. This fact can  make it  even more difficult to  define geometric invariants  and determine  the representative  Gram matrix in its equivalent class.

 By mainly using  quaternionic Cartan's angular  invariant and quaternionic cross-ratio in  $\overline{{\bf H}_\bh^n}$,  the author  \cite{caogd16}  solved  moduli problems  of  $\mathcal{M}(n,3; \overline{{\bf H}_\bh^n})$ and  $\mathcal{M}(n,4; \bhq)$, respectively.

We will continue the research in this direction. In this paper we concentrate  on the moduli problems  of  $\mathcal{M}(n,m; \bhq)$ and $\mathcal{M}(n,m; \bp(V_+) )$.   As stated in \cite{cungus12, cungus12fr}, the motivation of  our concerns comes from the research topic of  deformation spaces of pure loxodromic subgroup, as well as the current hot research topic concerning  subgroup generated by reflections in  submanifolds of dimension $n-1$ in ${\bf H}_\F^n$.

We need several notations to illustrate  our strategies  for overcoming the difficulties mentioned above.
At first, we figure out the relationship  between  the Gram matrix $G(\p)$ and  that of its permutation $\sigma(\p)$.  Using this  relationship, we are free to rearrange  the ordered $m$-tuple in question.

The  elementary matrix obtained by swapping row $i$ and row $j$ of the identity matrix $I_m$ is denoted by $T_{ij}$.   Let $\sigma$ be an element of symmetric group $\mathcal{S}_m$.  It is  well-known that $\sigma$  can be  expressed as  the product of transpositions $\sigma=\sigma_1\sigma_2\cdots \sigma_l$. We denote $T_{\sigma_t}=T_{ij}$ if $\sigma_t$ is a transposition of $i\to j\to i$ and  define
\begin{equation*}
T_{\sigma}=T_{\sigma_1}\cdots T_{\sigma_l}.
\end{equation*}

We can easily verify the following proposition.
\begin{prop}\label{permu-Gram}
	Let  $\p=(\p_1,\cdots,\p_m)$ be an  $m$-tuple of  points in $\F^{n,1} -\{0\}$.  Let $\sigma$ be an element of symmetric group $\mathcal{S}_m$.  Let $\sigma(\p)=(\p_{\sigma(1)},\cdots,\p_{\sigma(m)})$.    Then \begin{equation*}
	G(\p)=T_{\sigma}G(\sigma(\p))T_{\sigma}^*.
	\end{equation*}
\end{prop}

Let $v=(v_1,\cdots,v_t)$ be a row vector in $\bh^t$ and \begin{equation*}{\bf O}_v=\{\mu v\mu^{-1}=(\mu v_1 \mu^{-1},\cdots,\mu v_t \mu^{-1}): \forall \mu\in {\rm Sp}(1)\}.\end{equation*}
The set ${\bf O}_v$ can be thought of   as the orbit of $v$ under the action of ${\rm Sp}(1)/\pm 1$. The procedure of giving  a coordinate to the orbit ${\bf O}_v$  is termed by {\it rotation-normalized algorithm} in this paper.
  We mention that rotation-normalized algorithm  stems both from the noncommutativity of  quaternions and the existence of
 isometries of the form $\mu I_{n+1}$ in $\spn$.
 Such an algorithm is indigenous in  quaternionic  hyperbolic geometry, while  obviously  vacuous in complex hyperbolic geometry.  We mention that  rotation-normalized algorithm is involved in each moduli problem of quaternionic hyperbolic geometry.

 When $V$ is parabolic, the Gram matrix $G(\p)$  loses  the information  of configuration and  only carries the information of strati-form structure (see Example  \ref{lostinf} and Proposition \ref{deg-case}).
This strati-form structure will help us to break down the space  $V={\rm span}\{\p_1,\cdots, \p_m\}$ into finite $2$-dimensional subspaces.  We mention that there exist  at most   $n-1$  such $2$-dimensional subspaces in $\F^{n,1}$. These $2$-dimensional subspaces share a common basis which is a fibre in $V_0$.  In each  subspace containing  more than three  points  of  the  $m$-tuple, we need to introduce new invariants (the cross-ratios in $\bh\cup \infty$) to parameterize their congruence classes.  Of particular interest will be the harmonious coexistence  of these  $2$-dimensional subspaces (see Proposition \ref{wholepoints}).

When $V$ is not  parabolic, the Gram matrix $G(\p)$ contains the full information of   the congruence class of  $\p$.  The moduli problem  on  $\bp(V_0)$ is tractable for each entry in  Gram matrix $G(\p)$ being nonzero.
 On handling the moduli problem  on  $\bp(V_+)$,  the pivotal  point  is  to find  a partition of  $S(m)=\{1,\cdots,m\}$ to perform rotation-normalized algorithm  in each block independently.  This will help us to   tackle the difficulty caused by  orthogonality.  Such a method is termed by {\it block-normalized algorithm}.

In our perspective, the parameter  of   $\pspn$-congruence class of  $\p$ is    independent entries of a unique representative Gram matrix when $V$ is not  parabolic. For example,  the  $\pspn$-congruence class of  three points in $\bhqt$    is its quaternionic Cartan's angular  invariant \cite{apakim07,caogd16}.
 We mainly rely on  the rotation-normalized  and block-normalized algorithms to construct such a moduli space in this paper.   Our approaches  sound natural and elementary.

  Of course, one can construct other geometric invariants based on the independent  entries of the unique Gram matrix, and search a bijective map between them.  These geometric meanings of  these  invariants may help us to understand the configuration of  points in  $\p$.    These efforts may be involved in using Hermitian product in  more positions to detour the pitfalls caused by orthogonality among positive points.  We will not concentrate on that aspect  in the present paper.

As should be apparent, our ideas and exposition owe a great deal to the works  of  the references cited above, especially to those of  \cite{cungus12,cungus12fr}.

The paper is organized  as follows. Section \ref{sec2sig} contains  properties of quaternions, the some basic facts in  quaternionic hyperbolic geometry and the inertia of Gram matrices. These properties provide us with the tool  to execute rotation-normalized algorithm and initiate the idea of  block-normalized algorithm.    Section \ref{sec3mv0} describes the moduli problem on $\bp(V_0)$ for $m>4$. This may be regarded as  a generalization of that of  \cite{caogd16}, or the counterpart in quaternionic geometry  of that of \cite{cungus12}. The application of rotation-normalized algorithm is fully described. This method will be  mimicked in the more complicated cases in succeeding sections.
  Section \ref{sec4m2} is devoted to  describing the duality of submanifolds of dimension $n-1$ and the polar vectors.  The  parameter space of  $\mathcal{M}(n,2; \bp(V_+) )$ is also  constructed.
 In Section \ref{sec5v+},  we mainly refine the structure of Gram matrices.  These refined structures are crucial in introducing  new invariants in non regular case and the block-normalized algorithm in regular case.
  In Section \ref{sec6ireg},  we construct
invariants which describe the $\pspn$-congruence classes of $V$  when  $V$ is parabolic.
 In Section \ref{sec7reg}, we describe the moduli space of configurations of quaternionic $(n-1)$-dimensional submanifolds  when  $V$ is  not parabolic in conceptual style.    Section \ref{sec8tri} contains  a parameter space of quaternionic hyperbolic triangles. The content of this  section may be regarded as an application of somewhat conceptual results in previous sections in  hyperbolic  triangle groups, a current hot research topic in  hyperbolic geometry.

Shortly after we completed this paper,   Gou  informed us that  He has also considered  similar problem in  the boundary of quaternionic hyperbolic space  \cite{gao}.

\section{The inertia of  Gram matrices}\label{sec2sig}
In this section, we will recall  some  properties of quaternions  and obtain some properties of  the inertia of Gram matrices.

\subsection{ Properties of quaternions}\label{sbs-quat}
Recall that a  quaternion is of the form $a=a_0+a_1{\bf i}+a_2{\bf j}+a_3{\bf k}\in \bh$
where $a_i\in \br$ and $ {\bf i}^2 = {\bf j}^2 = {\bf k}^2 = {\bf
	i}{\bf j}{\bf k} = -1.$ Let $\overline{a}=a_0-a_1{\bf i}-a_2{\bf
	j}-a_3{\bf k}$ and $|a|= \sqrt{\overline{a}a}=\sqrt{a_0^2+a_1^2+a_2^2+a_3^2}$  be the  conjugate  and modulus of $a$, respectively.  We define $\Re(a)=(a+\overline{a})/2$ and $\Im(a)=(a-\overline{a})/2$.    Two quaternions $a$ and $b$ are similar if there exists nonzero $\lambda \in \bh$  such that $b=\lambda a
\lambda^{-1}$.

It is useful to  view $\bh$ as  $\bh=\bc\oplus  \bc {\bf j}$.  In this way, each quaternion $a=a_0+a_1{\bf i}+a_2{\bf j}+a_3{\bf k}$ can be uniquely expressed as
$$a=(a_0+a_1{\bf i})+(a_2+a_3{\bf i}){\bf j}=c_1+c_2 {\bf j}=c_1+{\bf j}\bar{c_2}.$$
It is well-known that the action of ${\rm Sp}(1)/\pm 1$ on $\bh$ coincides with the action of ${\rm SO}(3)$ on $\br^3$.  We recall it as the following proposition.
\begin{prop}\label{so3}
Denote $\overrightarrow{v}=(x,y,z)^T$  for
$v=x{\bf i}+y{\bf j}+z{\bf k}\in \bh$, where $A^T$  is the
transpose of matrix $A$.
For a unit quaternion $\mu=u_0+u_1{\bf i}+u_2{\bf j}+u_3{\bf k}$,  we  define
$$M_{\mu}=\left(\begin{array}{ccc}
u_1^2+u_0^2-u_3^2-u_2^2&     2u_1u_2+2u_0u_3&     2u_1u_3-2u_0u_2\\
2u_1u_2-2u_0u_3 & u_2^2-u_3^2+u_0^2-u_1^2&     2u_2u_3+2u_0u_1 \\
2u_1u_3+2u_0u_2&     2u_2u_3-2u_0u_1& u_3^2-u_2^2-u_1^2+u_0^2
\end{array}\right).$$
Then  $M_{\mu}\in {\rm SO}(3)$ and $$\overrightarrow{\bar{\mu}v\mu}=M_{\mu}\overrightarrow{v}.$$
 In particular  $$M_{e^{{\bf i}\beta}}=
\left(\begin{array}{ccc}
1&    0&     0\\
0 & \cos 2\beta&    \sin 2\beta \\
0&    -\sin 2\beta& \cos 2\beta
\end{array}\right).$$
\end{prop}

\begin{lem}\label{lemtonorm}
	Let $v_1=x_1{\bf i}+y_1{\bf j}+z_1{\bf k}$ and  $v_2=x_2{\bf i}+y_2{\bf j}+z_2{\bf k}$  such that  $\overrightarrow{v_1}$ and $\overrightarrow{v_2}$  are linear independent.  Let  $ v_1\cdot v_2=\overrightarrow{v_2}^T\overrightarrow{v_1}$.   Then there exists a unique  element  $\mu\in {\rm Sp(1)/\pm 1}$  such that
	\begin{equation}\label{normmu}\bar{\mu}v_1\mu=|v_1|\bi,\  \bar{\mu}v_2\mu=\frac{v_1\cdot v_2}{|v_1|}\bi+ \frac{\sqrt{(|v_1||v_2|)^2-(v_1\cdot v_2)^2}}{|v_1|}\ \bj.
	\end{equation}
\end{lem}

\begin{proof}
	Let $v_1=x_1{\bf i}+y_1{\bf j}+z_1{\bf k}$,   $v_2=x_2{\bf i}+y_2{\bf j}+z_2{\bf k}$ and  $\theta$  the angle between $\overrightarrow{v_1}$ and $\overrightarrow{v_2}$.
Identify $\Im(\bh)$ with  the $3$-dimensional real space ${\bf xyz}$.
Geometrically, by rotating the plane spanned by $v_1$ and $v_2$ to ${\bf xy}$ plane and then rotating around the $\z$-axis or ${\bf x}$-axis if necessary,
 we can obtain a $\mu$ such that formulae (\ref{normmu}) hold.
	It is helpful to regard  this formulae   as 	$$\bar{\mu}v_1\mu=|v_1|\bi, \bar{\mu}v_2\mu=|v_2|\cos\theta \bi+|v_2|\sin\theta \bj. $$
	Suppose that there exists another unit quaternion $\nu$ satisfying the above equalities. Then we have $\nu^{-1}\mu|v_1|\bi \bar{\mu}\bar{\nu}^{-1}=|v_1|\bi$ and therefore $\nu^{-1}\mu $ is a unit complex number. Similarly we get $\nu^{-1}\mu \bj \bar{\mu}\bar{\nu}^{-1}=\bj$ which implies that $\nu^{-1}\mu=\pm 1$.  Therefore $\nu=\mu$ or $\nu=-\mu$.
\end{proof}

Lemma \ref{lemtonorm}  is the foundation of  rotation-normalized algorithm. We give an  explicit formula of such a unique $\mu$ by the following process.
Note that $$-(|v_1|\bi+v_1)v_1(v_1|\bi+v_1)=\big||v_1|\bi+v_1\big|^2 |v_1|\bi.$$
Let \begin{equation}\label{findnu}\nu=\nu(v_1)=\left\{
\begin{array}{ll}
{\bf j}, & \hbox{provided}\ x_1<0, y_1^2+z_1^2=0; \\
\frac{|v_1|\bi+v_1}{\sqrt{2|v_1|(|v_1|+x_1)}}, & \hbox{otherwise.} \\
\end{array}
\right.\end{equation}
Then  $$|\nu|=1,  \bar{\nu}v_1\nu=|v_1| \bi.$$ Let $\bar{\nu}v_2\nu=c_1+c_2\bj$, where $c_1,c_2$ are complex numbers.  Since $c_2\neq 0$,  we have  $e^{-2{\bf i}\alpha} c_2=|c_2|$ with   $e^{{\bf i}\alpha}=\sqrt{\frac{c_2}{|c_2|}}$.    Therefore  $\mu=\pm \nu e^{{\bf i}\alpha}$ is the desired unit quaternion.
By finding the corresponding $c_2$ and (\ref{findnu}), we obtain  the following formula:
\begin{equation}\label{findmu}\mu=\mu(v_1,v_2)=\left\{
\begin{array}{ll}
\pm\sqrt{\frac{y_2+z_2\bi}{\sqrt{y_2^2+z_2^2}}}\ \bj, & \hbox{provided}\ x_1<0, y_1^2+z_1^2=0; \\
\pm\frac{|v_1|\bi+v_1}{\sqrt{2|v_1|(|v_1|+x_1)}}\ \sqrt{\frac{F}{|F|}} , & \hbox{otherwise,} \\
\end{array}
\right.
\end{equation}
where $$F=2x_2(|v_1|+x_1)(y_1+z_1\bi)-(|v_1|+x_1)^2(y_2+z_2\bi)+(y_2-z_2\bi)(y_1+z_1\bi)^2.$$

\subsection{The inertia of  Gram matrices}\label{sec2.3ine}
In this paper,  the  $J$  in  quaternionic Hermitian product  $\langle\z,\,\w\rangle=\w^*J\z$ given in Section \ref{se1-intr} will be taken one of the following  forms:
$$J_b=\left(
\begin{array}{cc}
I_n & 0 \\
0 & -1 \\
\end{array}
\right)\ \mbox{or}\  J_s=\left(
\begin{array}{ccc}
0 & 0 & 1 \\
0 & I_{n-1} & 0 \\
1 & 0 & 0\\
\end{array}
\right).$$
The corresponding quaternionic hyperbolic spaces are usually termed by {\it ball model} and {\it Siegel domain model}, respectively.
Let $C$ be the Cayley transformation  mapping  the ball to the  Siegel domain.  Then the relation of the two models can  be  mainly expressed by the following two equations:  $$\w^*J_b\z=(C\w)^*J_s(C\z),\, \,  g^*J_bg=J_b=C^{-1}J_sC=C^{-1}(CgC^{-1})^*J_s(CgC^{-1})C.$$
Each model has its own advantage in some situations.   Basically we  work on Siegel domain model only  in Sections \ref{sec6ireg}.

Note that $g^*J_bg=J_b$ with $g=(g_1,\cdots,g_{n+1})$, that is,
\begin{equation}\label{chaele}\langle g_i, g_j\rangle=0, i\neq j, \langle g_i, g_i\rangle=1, i=1,\cdots,n, \langle g_{n+1}, g_{n+1}\rangle=-1.\end{equation}
In  terms of   Gram matrix given by Definition \ref{dfngram},  we have $$G(g)=J_b, \forall g\in \spn.$$  Based on this observation, we  have  the following proposition.
\begin{prop}\label{mutuorth}
	Let   $\p=(\p_{1},\cdots, \p_{m})$ and   $\q=(\q_{1},\cdots, \q_{m})$  such that $\langle \p_{i}, \p_{i}\rangle=\langle \q_{i}, \q_{i}\rangle=1$   and $\langle\p_{i}, \p_{j}\rangle=\langle\q_{i}, \q_{j}\rangle=0, i\neq j$. Then there is a    $g\in \spn$ such that $$g\p_i=\q_i,\  i=1,\cdots,m.$$
\end{prop}
\begin{proof}
	By the signature restriction, we have $m\le n$. We can extend $\p$ and $\q$ to $f=(\p,\p_{m+1},\cdots,\p_{n+1})$ and  $h=(\q,\q_{m+1},\cdots,\q_{n+1})$ such that $f,h\in \spn$. Then $g=hf^{-1}$ is the desired isometry.
\end{proof}
Proposition \ref{mutuorth}  implies   the following simple result.
\begin{thm}\label{1trans}
	$\pspn$ acts  transitively  on $\bp(V_+)$.
\end{thm}

Let $\z^{\perp}=\{\w\in \bh^{n,1}:  \langle \z, \w\rangle=0\}$  be the orthogonal complement of the  fibre $\z\bh$  in $\bh^{n,1}$ and  $\dim_q(V)$  the quaternionic dimension of  subspace $V$ of  $\bh^{n,1}$.

\begin{prop}\label{orthcomp} We have the following statements concerning the orthogonal complements on $\bh^{n,1}$.
	\begin{itemize}
		\item [(i)] If $\z\in V_-$ then $\z^{\perp}\subset V_+$.  There exists  an orthogonal basis $\{\p_2,\cdots, \p_{n+1}\}$ in  $\z^{\perp}$,  $\dim_q(\z^{\perp})=n$ and $\{\z,\p_2,\cdots, \p_{n+1}\}$ is a basis of  $\bh^{n,1}$.
		\item [(ii)]If $\z\in V_0$ then $\z^{\perp}\subset V_+\cup V_0$ and $\z^{\perp}\cap V_0=\z\bh$. There exist  mutually orthogonal vectors  $\{\p_2,\cdots, \p_{n}\}$  in $ V_+$  and  $$\z^{\perp}={\rm span}\{\z,\p_2,\cdots, \p_{n}\}.$$
		\item [(iii)]If $\z\in V_+$ then  $$\z^{\perp}\cap V_+\neq \emptyset,\ \z^{\perp}\cap V_0\neq \emptyset, \z^{\perp}\cap V_-\neq \emptyset.$$  There exist  mutually orthogonal vectors  $\{\p_2,\cdots, \p_{n},\p_{n+1}\}$   such that   $$ {\rm span}\{\z,\p_2,\cdots, \p_{n}\}\subset V_+,\  \p_{n+1}\in V_-$$  and $\{\z,\p_2,\cdots, \p_{n+1}\}$ is a basis of  $\bh^{n,1}$.
	\end{itemize}
\end{prop}
\begin{proof}
	Let  $\z\in V_- $. Then $\z^{\perp}\subset V_+$. By (\ref{chaele}), there exists   an orthogonal basis $\{\p_2,\cdots, \p_{n+1}\}$ in  $\z^{\perp}$. Hence   $\dim_q(\z^{\perp})=n$ and $\{\z,\p_2,\cdots, \p_{n+1}\}$ is a basis of  $\bh^{n,1}$.  Therefore  case  (i)  holds.  Case (iii) follows similarly.
	
	Let  $\z\in V_0$. We may assume that $\z=(1,0,\cdots,0,1)^T$. It is obvious that $\w \in \z^{\perp}$ is of the form $\w=(q_1,q_2,\cdots,q_{n},q_1)^T$. Let $\e_i$ be the standard basis of  $\bh^{n,1}$. Then  $\e_i,i=2,\cdots,n$ belong to $\z^{\perp}$ and $$\z^{\perp}={\rm span}\{\z,\e_2,\cdots, \e_{n}\}.$$
\end{proof}

Recall that  $A\in \bh_{n,n}$ is called Hermitian if and only if $A=A^*$.   Let ${\rm H}_n(\bh)$   be the collection  of $n\times n$ Hermitian matrices.
It is well-known that the  right eigenvalues of  $A\in {\rm H}_n(\bh)$ are real and there exists an invertible matrix $B\in \bh_{n,n}$  such that $B^*AB$ is a diagonal matrix which has only entries $+1,-1,0$ along the diagonal.  The numbers of $+1$s, $-1$s and $0$s are denoted by $n_+,n_-$ and $n_0$, respectively. We denote the signature of $A$ by $$i(A)=(n_+,n_-,n_0).$$

\begin{prop}(\cite[Proposition 1.1]{caogd16})\label{prop-1.1} If $\z,\w\in \bh^{n,1}-\{0\}$  with $\langle \z,\,\z \rangle\leq 0 $ and  $\langle\w,\,\w\rangle\leq 0 $ then either $\w=\z\lambda$ for some $\lambda\in \bh$ or  $\langle\z,\,\w\rangle\neq 0$.
\end{prop}
	
\begin{prop}\label{containneg}
		Let  $\mathfrak{p}=(p_1,\cdots,p_m)$ be an  $m$-tuple of pairwise distinct points in $\partial {\bf H}_{\bh}^n$ with lift  $\p=(\p_{1},\cdots, \p_{m})$ and $m\geq 2$.   Then  $G(\p)$ has a negative eigenvalue.
\end{prop}

\begin{proof}
	Let $\q=\p_1+\p_2\mu$ with  $\mu=-\langle\p_1,\p_2 \rangle$.  By Proposition \ref{prop-1.1},  \begin{equation}\label{negel}\langle\q,\q \rangle=-2|\langle\p_1,\p_2 \rangle|<0.\end{equation}
	Suppose that the eigenvalues of $G(\p)$ are all non-negative. Then there exists an invertible matrix $S\in \bh_{m,m}$ such that $$ S^*G(\p)S={\rm diag}(1,\cdots,1,0,\cdots,0).$$
	Then $x^*S^*\p^*J\p Sx\geq 0, \forall x\in \bh^m$.  This contradicts  (\ref{negel})  with $x=S^{-1}l$ and $l=(1,\mu,0,\cdots,0)^T\in \bh^m$.
\end{proof}

The following proposition is obvious.
\begin{prop}\label{pro-inertia}
	Let $S$ be an invertible matrix. Then $i(A)=i(S^*AS)$.  Furthermore  assume that  $S^*AS=\left(
	\begin{array}{cc}
	A_1 & 0\\
	0 & A_2
	\end{array}
	\right)$. Then  $$i(A)=i(A_1)+i(A_2).$$
\end{prop}
Let $\p=(\p_1,\cdots,\p_l)$ and $\q=(\q_1,\cdots,\q_t)$  such that  $\langle \p_i, \q_j\rangle=0$ for all $i,j$.
Then
\begin{equation}\label{orthcount}
(\p,\q)^*J(\p,\q)=\left(
\begin{array}{cc}
G(\p) & 0\\
0 &G(\q)
\end{array}
\right).
\end{equation}

We can now prove the following crucial  result.
\begin{thm}\label{thm-inertia}
	Let $\p=(\p_1,\cdots, \p_m)\in  \bh_{n+1,m}$,  $V={\rm span}\{\p_1,\cdots, \p_m\}$
	and  $$\dim_qV=k+1,\ i(G(\p))=i(\p^*J\p)=(n_+,n_-,n_0).$$ Then $$k\leq n_++n_-\leq k+1,n_+\le  n, \ n_-\le 1 , \  n_++n_-+n_0=m.$$
	In particular, we have the following statements.
		\begin{itemize}
		\item[(1)]  If $\p_i\in V_0, i=1,\cdots,m$	 then $n_+=k,\ n_-=1$.
		
				\item[(2)]    If $\p_i\in V_+,i=1,\cdots,m $	 then there are three cases:
		\begin{itemize}
			\item[(i)]  $n_+=k,\ n_-=1$,   in this case  $V$ is hyperbolic;
			
			\item[(ii)]  $n_+=k+1,\ n_-=0$,   in this case  $V$ is elliptic;
			\item[(iii)]$n_+=k,\ n_-=0$,   in this case  $V$ is parabolic.
		\end{itemize}
	\end{itemize}
\end{thm}

\begin{proof}
	Let  $t=k+1$. Without loss of generality, we assume that $\p_1,\cdots, \p_t$ are linearly independent  and $$\p_j=\p_1\lambda_{1j}+\cdots+\p_t\lambda_{tj}, j=t+1,\cdots,m.$$
	Let $\q=(\p_1,\cdots, \p_t)\in  \bh_{n+1,t}$. Then $\p=\q(I_t,\Lambda)$, where $\Lambda=(\lambda_{ij}), i=1,\cdots t, j=t+1,\cdots,m$. Let $S=\left(
	\begin{array}{cc}
	I_t & -\Lambda\\
	0 & I_{m-t}
	\end{array}
	\right).$ Direction computation shows that   $$S^*G(\p)S=S^*\p^*J\p S=\left(
	\begin{array}{cc}
	\q^*J\q & 0\\
	0 & 0
	\end{array}
	\right).$$
	Therefore, by Proposition \ref{pro-inertia}  we have that  $$i(\p^*J\p)=i(\q^*J\q).$$  This implies that $ n_++n_-\leq k+1$.

	If  $V\cap V_-\neq \emptyset$ then there exists a   $\z\in V_-$ such that $V=\z\bh\oplus (\z^{\perp}\cap V)$.  In the  space $\z^{\perp}\cap V$  there exist $k$ mutually  orthogonal  positive lines $\q_1,\cdots,\q_k$ such that $V={\rm span}\{\z,\q_1,\cdots, \q_k\}$.  By (\ref{orthcount}) we have  $n_+=k,\ n_-=1$ and  $V$ is   hyperbolic in this case.
	
	 By Proposition \ref{containneg},  a space with two different  null lines must contain negative lines.   If  $V\cap V_-=\emptyset$  and $V\cap V_0\neq \emptyset$  then there exists a unique  $\z\bh\in V_0$.	  The space $\z^{\perp}\cap V$ contains only $k$   mutually  orthogonal    positive lines $\q_1,\cdots,\q_k$.  In this case  $n_+=k,\ n_-=0$  and  $V$ is    parabolic.
	
	If  $V\subset V_+$, then $V$ contains $k+1$   mutually  orthogonal  positive lines $\q_1,\cdots,\q_{k+1}$.   In this case  $n_+=k+1,\ n_-=0$  and  $V$ is    elliptic.

It follows from  Proposition \ref{orthcomp} and \ref{containneg} that  the statements of (1) and (2) hold.
\end{proof}

\begin{rem}\label{limdim}
Since any $m$-tuple  $\p=(\p_1,\cdots, \p_m)$ in $\F^{n,1}$ span a space $V={\rm span}\{\p_1,\cdots, \p_m\}$ which  is definitely  contained in  a copy of $\F^{m,1}$.  In other words, there exists a  $g\in \spn$ such that $g(V)\subset \F^{m,1}\hookrightarrow\F^{n,1}$. So if one consider the moduli problem of points in $\F\bp^n$, it is enough to  assume that $n\leq m$. Furthermore, for moduli problem of points in $\overline{{\bf H}_\F^n}$, one can further assume that $n\leq m-1$.
\end{rem}

\section{Moduli problem on $\bp(V_0)$}\label{sec3mv0}
In this section, we will  consider the moduli problem on $\bp(V_0)$ for $m >4$.
The application of rotation-normalized algorithm is fully described. This method will be  mimicked conceptually  to more complicated cases in Sections \ref{sec6ireg} and \ref{sec7reg}.

\subsection{Semi-normalized Gram matrix}
We recall the following definition in  \cite{apakim07,caogd16}.
\begin{dfn} \label{car32 angularh} The {\it quaternionic Cartan's angular  invariant} of a triple  $\fp=(p_1,p_2,p_3)$   of pairwise distinct points in $\overline{{\bf H}_{\bh}^n}$ is the  angular invariant $\ba_{\bh}(\fp)$,  $0 \leq \ba_{\bh}(\fp)\leq \frac{\pi}{2}$,    given by
	\begin{equation} \label{angular} \ba_{\bh}(\fp)=\ba_{\bh}(p_1,p_2,p_3):=\arccos \frac{\Re(-\langle \p_1, \p_2, \p_3\rangle)}{|\langle \p_1, \p_2, \p_3\rangle|},\end{equation}
	where $\p_1,\p_2,\p_3$ are lifts of $p_1, p_2, p_3$, respectively.
\end{dfn}

\begin{prop}\label{normprocessn}
	Let  $\mathfrak{p}=(p_1,\cdots,p_m)$ be an  $m$-tuple of pairwise distinct points in $\partial {\bf H}_{\bh}^n$.  Then the equivalence class of Gram matrices associated to $\mathfrak{p}$ contains a matrix $G= (g_{ij})$ with
	$$g_{ii} = 0,\ i=1,\cdots, m,\   g_{i-1,i}=1, i=2,\cdots, m, g_{13}=-e^{-{\bf i}\ba},$$  where $\ba=\ba_{\bh}((p_1,p_2,p_3))$.
\end{prop}

\begin{proof}  Let $\p=(\p_{1},\cdots, \p_{m})$ be an arbitrary lift of $\mathfrak{p}$.   We want to obtain  a diagonal matrix $D$  such that  $G(\p D)$  is the desired Gram matrix.

	Note that $\langle \p_{i}, \p_{j}\rangle\neq 0$ for $i\neq j$.   Firstly we  obtain the solutions  $\lambda_i, i=2,\cdots,m$   of  the equations below:
	\begin{equation}\label{findlambda} \langle {\bf p}_{1}, \ {\bf p}_{2}\lambda_2\rangle=1,\ \langle {\bf p}_{2} \lambda_2, \ {\bf p}_{3}\lambda_3\rangle=1, \cdots,  \langle {\bf p}_{m-1}\lambda_{m-1}, {\bf p}_{m}\lambda_m\rangle =1. \end{equation}
		Next,  by (\ref{findnu}) we  let \begin{equation}\label{flam1}\lambda_1=\frac{\nu( \langle {\bf p}_{1}, {\bf p}_{3}\lambda_3\rangle)}{\sqrt{| \langle {\bf p}_{1}, {\bf p}_{3}\lambda_3\rangle |}}=\frac{\nu(\langle {\bf p}_{2}, {\bf p}_{1}\rangle \langle {\bf p}_{2}, {\bf p}_{3}\rangle^{-1} \langle {\bf p}_{1}, {\bf p}_{3}\rangle)}{\sqrt{|\langle {\bf p}_{2}, {\bf p}_{1}\rangle \langle {\bf p}_{2}, {\bf p}_{3}\rangle^{-1} \langle {\bf p}_{1}, {\bf p}_{3}\rangle|}}.\end{equation}
	By the property of  quaternionic Cartan's angular  invariant,  $\langle {\bf p}_{1} \lambda_1, \ {\bf p}_{3}\lambda_3\lambda_1\rangle$ is a unit complex with negative real part and therefore
 $$\langle \p_{1}\lambda_1 , \ \p_{3}\lambda_3\lambda_1\rangle=-e^{-{\bf i}\ba}.$$
	Let $\mu_1=\lambda_1$; for $i\geq 2$,  $\mu_i=\lambda_i \lambda_1$ when $i$ is odd,  and $\mu_i=\lambda_i \overline{\lambda_1}^{-1}$ when $i$ is even.  Then  $G(\p D)$  is the desired Gram matrix with  $$D={\rm diag}(\mu_1,\cdots, \mu_m).$$
\end{proof}

\begin{dfn}
	The Gram matrix  $G$ as in Proposition \ref{normprocessn} of the form
	\begin{equation}\label{semimatrix}
	G({\bf n})=(g_{ij})=\left(
	\begin{array}{cccccc}
	0 & 1 & g_{13} & g_{14} & \cdots & g_{1m} \\
	1 & 0 & 1 & g_{24}&\cdots & g_{2m} \\
	\overline{g_{13}} & 1 & 0 & 1&\cdots & g_{3m}  \\
	\overline{g_{14}} & \overline{g_{24}} & 1 &  \ddots & \ddots &  \vdots\\
	\vdots &  \vdots  &  \vdots  &  \ddots & 0 &  1 \\
	\overline{g_{1m}} & \overline{g_{2m}} & \overline{g_{3m}} & \cdots& 1 & 0\\
	\end{array}
	\right)
	\end{equation}
	is called  the {\it  semi-normalized Gram matrix}.
\end{dfn}

\begin{prop}(\cite[Theorems 2.1, 2.2]{cungus12})\label{prop-condgn}
	Let $G= (g_{ij})$ be  a  Hermitian $m\times m$-matrix, $m>2$  with
	$$g_{ii} = 0,\ i=1,\cdots, m,\   g_{i-1,i}=1, i=2,\cdots, m, g_{13}=-e^{-{\bf i}\ba},$$   where $\ba\in[0,\pi/2]$.  Let $i(G)=(n_+,n_-,n_0)$.
	Then $G$ is a  semi-normalized Gram matrix  associated with some ordered  $m$-tuple $\fpm$
	of pairwise distinct isotropic points in $\bhq$  if and  only if
	\begin{equation}\label{inercond}n_+\le  n, \ n_-=1 , \  n_++n_-+n_0=m.\end{equation}
\end{prop}

\begin{proof}
	Suppose that  $G$ is a  semi-normalized Gram matrix  associated with some ordered  $m$-tuple $\fpm$ of  pairwise distinct isotropic points in $\bhq$.   It follows from  Theorem \ref{thm-inertia} that $n_+\le  n, \ n_-=1 , \  n_++n_-+n_0=m$.
	
	Conversely,  suppose that  $G= (g_{ij})$ is  of the form (\ref{semimatrix}) with
	$$i(G)=(n_+,1,m-n_+-1).$$
	There exists an invertible matrix $S$ such that $S^*GS=B$,
	where $B$ is the diagonal $m \times m$ matrix
	with $b_{ii} = 1$ for $1 \leq i\leq  n_+, b_{ii} =-1$ for $i = n_+ + 1$, and $b_{ij} = 0$ for all other indices.  Now let $A = (a_{ij})$  be the $(n + 1) \times m$-matrix such that  $a_{ii} = 1$  for $1 \leq i\leq  n_+, a_{ii} =-1$ for $i = n_+ +1$, and $a_{ij}= 0$  for all other indices.  Then  $A^*JA=B=S^*GS$,  which implies that $$(S^*)^{-1}A^*JAS^{-1}=G.$$   Then  $\p=AS^{-1}$ is the desired lift of $\fpm$  to get the semi-normalized Gram matrix $G$.
\end{proof}

\subsection{The parameter  space of  moduli space}\label{submpv0}
The following lemma shows that a  semi-normalized Gram matrix is just an equivalent class, and also indicates the necessity of   performing  rotation-normalized algorithm.
\begin{lem}\label{lem-actsemi}
	Suppose that  the Gram matrix  $G(\p )$  is a semi-normalized Gram matrix for  $\p=(\p_{1},\cdots, \p_{m})$.   Then   $G(\p D)$  is  still a semi-normalized Gram matrix  with  $D={\rm diag}(\mu_1, \cdots, \mu_m)$  if  only if
	$$  D=\mu I_m={\rm diag}(\mu, \cdots, \mu),  \mu e^{-{\bf i}\ba}=e^{-{\bf i}\ba}\mu,   \mu\in {\rm Sp}(1).$$
\end{lem}

\begin{proof}
	It follows from $$\langle \p_{i-1}\mu_{i-1}, \ \p_{i}\mu_i\rangle=1, \  i=2,\cdots,m$$ that
 all those   $\mu_i$ with  $i$ odd   are equal,  and so do  for all  those $\mu_i$ with  $i$ even.
The fact $\langle \p_{1}\mu_1, \ \p_{3}\mu_3\rangle=-e^{-{\bf i}\ba}$ implies $\mu_1=\mu_3$.  Hence
	 $\mu_1=\mu_2=\cdots=\mu_m:=\mu$ and  $\mu e^{-{\bf i}\ba}=e^{-{\bf i}\ba}\mu $.
\end{proof}

Set $t=\frac{(m-1)(m-2)}{2}$.
We can represent a semi-normalized Gram matrix by a $t$-vector:
\begin{equation}\label{gtovect}v_G=(g_{13}, g_{14},g_{24}, \cdots, g_{1m},\cdots, g_{m-2,m}).\end{equation}
Also we represent \begin{equation}\label{2bijections}G=G(v_G).\end{equation}  Recall that  two Hermitian matrices $H$ and $\tilde{H}$ are equivalent if there exists a diagonal matrix $D$ such that $\tilde{H} = D^*HD$ (see \cite{caogd16,cungus12}).
By Lemma \ref{lem-actsemi}, we  obtain the following result.
\begin{lem}\label{lemvector}
	Let $G$  and $\tilde{G}$ be two semi-normalized Gram matrices represented by $V(G)$ and $V(\tilde{G})$.   Then $\tilde{G}$ and $G$  are equivalent  if and only if \begin{equation}\label{eqrelation}{\bf O}_{v_G}= {\bf O}_{v_{\tilde{G}}}.\end{equation}
\end{lem}

 From this,  Proposition \ref{prop-condgn} can be reformulated  as follows.

\begin{prop}\label{prop-condv}
	Let  $v=(v_1,\cdots, v_t)$ with  $v_1=-e^{-{\bf i}\ba}, \ba\in[0,\pi/2]$.   Let $i(G(v))=(n_+,n_-,n_0)$.
	Then $G(v)$ is a  semi-normalized Gram matrix  associated with some ordered  $m$-tuple $\fpm$
	of distinct isotropic points in $\bhq$  if and  only if
	\begin{equation}\label{inercondv}n_+\le  n, \ n_-=1 , \  n_++n_-+n_0=m.\end{equation}
\end{prop}

\begin{dfn}
	$$V(n,m)=\{v=(v_1,\cdots,v_t):  i(G(v))=(n_+,n_-,n_0) \ \mbox{with} \  n_+\le  n, \ n_-=1  \}.$$
\end{dfn}

By Lemma \ref{lemvector} there is  an equivalent relation in $V(n,m)$ defined by  (\ref{eqrelation}).
Therefore the configuration space $\mathcal{M}(n,m;\partial{\bf H}_\bc^n)$  can be thought of as the quotient
of $V(n,m)$ under this equivalent relation.  That is  $$\mathcal{M}(n,m;\partial{\bf H}_\bc^n)=V(n,m) /\simeq.$$
Based on this observation, we are ready to construct the  parameter space $\bm(n,m)$  for $V(n,m) /\simeq$ with rotation-normalized algorithm.   We mainly rely on  Lemma \ref{lemtonorm} to execute rotation-normalized algorithm.

This procedure can be described conceptually as follows:

In  case  $\ba=0$, or equivalently, $- e^{-{\bf i}\ba}=-1$, we basically need to find two entries  $v_i$ and $v_j$  in $v\in  V(n,m)$   with  $\Im(v_i)$ and $\Im(v_j)$ being linearly  independent  to specific the parameters for its representing equivalent class, whilst only a quaternion in $\bh-\bc$   in the case of $\ba\neq 0$.

  The above conceptual description  is a motivation of  the definition of the following sets.

Let $$\br^{2+}=\{v\in \bh:v=x_0+x_1\bi+x_2\bj, x_2>0\}, \br^{1+}=\{v\in \bc:v=x_0+x_1\bi, x_1>0\}.$$
\begin{dfn}
	We define the following sets. $$P(\bc)=\{v\in V(n,m):  v_1\notin \br, v_i\in \bc,\  \mbox{for}\ i=2,\cdots,t\};$$
	 $$P(j)=\{v\in V(n,m):  v_1\notin \br, v_i\in \bc,\  \mbox{for}\  i<j, v_j\in\br^{2+}\}, j=2,\cdots,t;$$
		$$Z(\br)=\{v\in V(n,m):  v_i\in \br\   \mbox{for}\ i=1,\cdots,t\};$$
	$$Z(\bc,i)=\{v\in V(n,m):   v_t\in \br,\  \mbox{for}\  t<i, v_i\in\br^{1+}\}, i=2,\cdots,t;$$
	$$Z(i,j)=\{v\in V(n,m):  v_t\in \br, t<i,  v_i\in \br^{1+};\  v_t\in \bc, t<j, v_j\in\br^{2+}\},   j=2,\cdots,t, 2\leq i<j .$$
\end{dfn}
We remark that the sets defined above is roughly divided by  two cases:  $\ba\neq 0$ and $\ba=0$ . Each case is refined  according to the positions  in which  Lemma \ref{lemtonorm} acts.    Roughly speaking,  such a $Z(i,j)$ looks like
$$Z(i,j)=(\underbrace{-1,\br^*,\cdots,\br^*}_{i-1},\br^{1+}, \underbrace{\bc^*,\cdots,\bc^*}_{j-i-1},\br^{2+},\bh^*,\cdots,\bh^*).$$
Let $$P(n,m)=P(\bc)\cup P(j), Z(n,m)=Z(\br)\cup Z(\bc,j)\cup Z(i,j)$$
and $$\bm(n,m)=P(n,m)\cup Z(n,m).$$

\begin{thm}\label{thmparam}
	$\bm(n,m)$ is a parameter space of $V(n,m)/\simeq$.
\end{thm}
\begin{proof}
	Let $v=(v_1,\cdots,v_t)\in V(n,m)$, where $v_1=-e^{-{\bf i}\ba}$.  We define a map
	\begin{equation}\label{eqcoordinate}  \psi:{\bf O}_v\in V(n,m)/\simeq\to \bm(n,m) \end{equation}
	by the following steps:
	
	The equivalent class ${\bf O}_v$ with  $\ba\neq 0$  will be mapped to an element in $P(n,m)$.  It is obvious that  $\bar{\mu}v\mu\in  V(n,m)$ if and only if  $\mu\in U(1)$.   If  all  entries of $v$ are complex numbers,  then   ${\bf O}_v$ is represented by $v$ itself.  Equivalently, the parameter of ${\bf O}_v$ assigned by $\psi$ in $\bm(n,m)$ is $v$ which  belongs to $\in P(\bc)$. Otherwise, let  $j$  be the  smallest index among entries of $v$  such that  $v_j\in \bh-\bc$.   Let $\mu=\mu(\Im(v_1),\Im(v_j))$  given by (\ref{findmu}).  Therefore  ${\bf O}_v$ is assigned to the parameter  $\bar{\mu}v\mu$, which belongs to $P(j)$.
	
The equivalent class ${\bf O}_v$ with  $\ba=0$  belongs  to $Z(n,m)$.  More precisely, if  all  entries of $v$ are reals,  then   ${\bf O}_v$ is represented by $v$ itself  belonging to $Z(\br)$.   We  divide the remainder into two cases.  If  all  entries of $v$ are complex numbers  with  $i$  being the  smallest index  such that  $v_i\in \bc-\br$.   Let $\mu=\nu(\Im(v_i))$  be given by (\ref{findnu}).   Then we assign  ${\bf O}_v$ to  $\bar{\mu}v\mu$, which belongs to $Z(\bc,j)$.
	For the latter case, let $i$  be the  smallest index such that  $v_i\in \bc-\br$ and $j$  the  smallest index such that  $v_j\in \bh-\bc$.  Let $\mu=\mu(\Im(v_i),\Im(v_j))$.   Then we assign  ${\bf O}_v$ to  $\bar{\mu}v\mu$, which belongs to $Z(i,j)$.
	
	By Lemma \ref{lemtonorm} and  the construction of $P(n,m)$ and $Z(n,m)$ above, the map $\psi$ is bijection.
	Therefore $\bm(n,m)$ is a parameter space of $V(n,m)/\simeq$.
\end{proof}
\begin{thm}\label{mainthm}
	The configuration space $\mathcal{M}(n,m;\partial{\bf H}_\bc^n)$ is homeomorphic to $\bm(n,m)$
\end{thm}
\begin{proof}
		Let $m(\mathfrak{p})\in \mathcal{M}(n,m;\partial{\bf H}_\bc^n)$ be the point represented by $\mathfrak{p}=(p_1,\cdots, p_m)$.  We can get a semi-normalized Gram matrix  $G$ with arbitrary lift of  $\mathfrak{p}$.  Proposition  \ref{prop-condv} and Theorem  \ref{thmparam} imply that  we can  define a map
	\begin{eqnarray*}
		\tau: m(\fp)\in \mathcal{M}(n,m;\partial{\bf H}_\bc^n)\to  \psi(v_G)\in \bm(n,m).
	\end{eqnarray*}
	This map is a bijection. Such a map  is a homeomorphism  because  $\bm(n,m)$ has the topology  structure induced from $\bh^t$.
\end{proof}
We conclude this section by some remarks. Firstly, if we allow  $m=3$ in our process then we get the parameter of quaternionic Cartan's angular invariant $\ba$ ( in fact  a complex number   $-e^{-{\bf i}\ba}$); while the case of  $m=4$  is exactly the result in \cite{caogd16}.  Secondly  it seems  that  the parameters of  $m$-tuples  in $Z(\br)$, $Z(\br)\cup Z(\bc,i)\cup P(\bc)$ can be thought of  as  $m$-tuples living in a copy of  $\partial	{\bf H}_{\br}^n$  and $\partial	{\bf H}_{\bc}^n$, respectively.

\section{Moduli space on $\bp(V_+)$ of case $m=2$}\label{sec4m2}

In this section we will  describe the configuration of  two  submanifolds of dimension $n-1$.   The  author believe that this fact is  well-known  in quaternionic hyperbolic geometry.  However we  did not find  any proof of it in the literature.  The  parameter space of  $\mathcal{M}(n,2; \bp(V_+) )$ is also  constructed.

\subsection{The duality of submanifold  of dimension $n-1$ and  polar vector}\label{sec4sub1}
It follows from Proposition \ref{orthcomp} that ${\p}^{\perp}$ is an $n$-dimensional subspace of $\bh^{n,1}$ for any  vector ${\p}\in V_{+}$.
\begin{dfn}We define
	\begin{equation}\label{submanifold}
	l_{\p}=\bp\big({\p}^{\perp}\cap (V_0\cup V_-)\big)=\bp({\p}^{\perp})\cap \overline{
		{\bf H}_{\bh}^n}, \, \mbox{for}\ \  {\p}\in V_{+}. \end{equation}
	$l_{\p}$ is a totally geodesic submanifold with boundary in $\overline{
		{\bf H}_{\bh}^n}$, which is equivalent to  $\overline{
		{\bf H}_{\bh}^{n-1}}$.
\end{dfn}
We call  ${\p}\in V_{+}$ a polar vector of  $l_{\p}$. Sometimes we drop off  $V_0$ in (\ref{submanifold}), and call $l_{\p}$ an  $(n-1)$-submanifold  in $\hq$.
Also for each  $(n-1)$-submanifold $M$, we can find a vector  ${\p}\in V_{+}$  such that  $\p \bh$ is the unique fibre with the property $M \subset  \bp({\p}^{\perp})$.
Due to this duality, the configuration of $m$-tuples of distinct $(n-1)$-submanifolds  is  equivalent to the configuration of   $m$-tuples of pairwise distinct positive points.

As in \cite{park2010}, we define the angle $\theta\in [0,\pi/2]$  between any pair of intersecting   $(n-1)$-submanifolds $l_{\p_1}$ and $l_{\p_2}$ by
$$ \cos^2(\theta) = \frac{\langle \p_1,\p_2\rangle\langle \p_2,\p_1\rangle}{\langle \p_1,\p_1\rangle\langle \p_2,\p_2\rangle}.$$
This is clearly invariant under  quaternionic hyperbolic isometries.

We need a formula to calculate the distance between a negative point and  an  $(n-1)$-submanifold.
\begin{lem}(\cite[Corollary 7.7]{park2010})\label{distline}
	Let $z$ be any point of ${\bf H}_{\bh}^n$ with lift $\bf z$. Then \begin{equation}\label{metric}
	\cosh^2\frac{\rho(l_{\p},z)}{2}=1-\frac{\langle \bf z,\p\rangle\langle \p,\bf z\rangle}{\langle \bf z,\bf z\rangle\langle \p,\p\rangle}\geq 1.\end{equation}
\end{lem}
\begin{proof}
	Let $\Pi _{l_{\p}}$ be the orthogonal projection  from $\bh^{n,1}$ to ${\p}^{\perp}$.
	Then we can express a lift of $z$ as $\z=\Pi _{l_{\p}}(\z)\lambda +{\p}\mu$. Since $\langle \Pi _{l_{\p}}(\z), {\p}  \rangle=0$, we have  $ \langle \z, {\p}  \rangle=\langle {\p}, {\p}  \rangle\mu$,  $|\langle \z, \Pi _{l_{\p}}(\z)  \rangle|^2=|\lambda|^2|\langle \Pi _{l_{\p}}(\z), \Pi _{l_{\p}}(\z)\rangle|^2$ and  $$ \langle \z, \z  \rangle=|\lambda|^2\langle \Pi _{l_{\p}}(\z), \Pi _{l_{\p}}(\z)  \rangle+|\mu|^2\langle {\p}, {\p}   \rangle.$$
	Hence
	\begin{equation*}
	\cosh^2(\frac{\rho(l_{\p},z)}{2})=\cosh^2(\frac{\rho(\bp(\Pi _{l_{\p}}(\z)),z)}{2})=\frac{|\lambda|^2|\langle \Pi _{l_{\p}}(\z), \Pi _{l_{\p}}(\z)  \rangle|^2}{\langle \z, \z  \rangle\langle \Pi _{l_{\p}}(\z), \Pi _{l_{\p}}(\z)  \rangle}=1-\frac{\langle \z, {\p}  \rangle\langle {\p}, \z  \rangle}{\langle \z, \z  \rangle\langle {\p}, {\p}  \rangle}.
	\end{equation*}
\end{proof}

The configuration of   two positive lines in $V_+$  can be  described as follows.
\begin{thm}(\cite[Proposition 7.8]{park2010})\label{thm-twpo}
	Let $\p_1,\p_2$  be  two points in  $V_+$ with distinct projections in  $\bp(V_+)$ , $V={\rm span}\{\p_1,\p_2\}$ and $$t=\frac{|\langle \p_{1}, \p_{2}\rangle|}{\sqrt{\langle \p_{1}, \p_{1}\rangle \langle \p_{2}, \p_{2}\rangle}}.$$ Then we have the following statements.
	\begin{itemize}
		\item[(i)]   $|\langle \p_1, \p_2 \rangle |^2< \langle \p_1, \p_1 \rangle \langle \p_2, \p_2 \rangle$ if and only if  $V\subset V_+$.  In this case,  $$\p_1^{\perp}\cap \p_2^{\perp}\cap V_-\neq \emptyset$$
		and
		the angle  between  $l_{\p_1}$ and   $l_{\p_2}$ is $\arccos	 t$.
		\item[(ii)]  $|\langle \p_1, \p_2 \rangle |^2= \langle \p_1, \p_1 \rangle \langle \p_2, \p_2 \rangle$ if and only if there exists a unique fibre $\z\bh\in V_0$ such that  $$V\cap V_0=\z\bh, V\cap V_-=\emptyset.$$  In this case  $$\p_1^{\perp}\cap \p_2^{\perp}\subset V_+\cup V_0,\p_1^{\perp}\cap \p_2^{\perp} \cap V_0=(\p_1-\p_2)\bh,$$
		which implies that 	  $l_{\p_1}$ and   $l_{\p_2}$  intersect in a unique point in $\bhq$.
		\item[(iii)] $|\langle \p_1, \p_2 \rangle |^2> \langle \p_1, \p_1 \rangle \langle \p_2, \p_2 \rangle$ if and only if  $V\cap V_0\neq \emptyset, V\cap V_-\neq \emptyset $.  \\  In this case   $$(\p_1^{\perp}\cap  V_-)\cap (\p_2^{\perp}\cap  V_-)=\emptyset$$
		and
		$$\cosh \left(\frac{\rho(l_{\p_1},l_{\p_2})}{2}\right)=\cosh \left(\frac{\rho(\bp(\z),\bp(\w))}{2}\right)=t,$$
		where   $V\cap \p_1^{\perp}=\z\bh$ and  $V\cap \p_2^{\perp}=\w\bh$.
	\end{itemize}
\end{thm}
\begin{proof} By normalization and the transitivity of $\pspn$ on $\bp(V_+)$,  we may assume that  \begin{equation}\label{twonorm}\langle {\p}_1, {\p}_1\rangle=\langle {\p}_2, {\p}_2 \rangle=1, t=\langle {\p}_1, {\p}_2 \rangle\geq 0, \end{equation} where $\p_1=(0,1,0,\cdots,0)^T$ and  $\p_2=(x_1, \cdots,x_{n+1})^T$.   With the above assumption we have
	$$ t=x_2\geq 0,  \sum_{i=1}^n|x_i|^2-|x_{n+1}|^2=1.$$  Let $\bu=\p_1\lambda_1+\p_2\lambda_2\in V$.
	Then \begin{equation}\label{twopo}\langle \bu, \bu \rangle=|\lambda_1|^2+|\lambda_2|^2+2\Re(\bar{\lambda}_2\lambda_1)t.\end{equation}
	We need to consider the following three cases $t<1,=1,>1$, respectively.
	
	Note that  $t<1$ if and only if  $\langle \bu, \bu \rangle>0$.   This implies that  $V\subset V_+$  for $t<1$.   In this case, there exists a  $\z=(z_1,0,z_3,\cdots,z_n,1)^T\in \p_1^{\perp}\cap \p_2^{\perp}\cap V_-$ satisfying the following equation  $$\bar{x}_1z_1+\bar{x}_3z_3+\cdots+\bar{x}_nz_n-\bar{x}_{n+1}=0.$$   In fact    $\bp(\p_1^{\perp}\cap \p_2^{\perp}\cap V_-)$ is equivalent to  ${\bf H}_{\bh}^{n-2}$ and the angle  between   $l_{\p_1}$ and   $l_{\p_2}$ is $\arccos t$.
	
	Observe that    $\langle \bu, \bu \rangle\geq 0$  for    $t=1$. Note that $\langle \bu, \bu \rangle=0$ if and only if $\lambda_2=-\lambda_1$.   Therefore  $(\p_1-\p_2)\bh$ is the unique fibre  in $V\cap V_0$.  It is obvious that $\langle \p_1, (\p_1-\p_2) \rangle=\langle \p_2, (\p_1-\p_2)\rangle=0$.  Each  $\z$  in $ \p_1^{\perp}\cap \p_2^{\perp}$ is of the form $(z_1,0,z_3,\cdots,z_{n+1})^T$  satisfying the following equation  $$\bar{x}_1z_1+\bar{x}_3z_3+\cdots+\bar{x}_nz_n-\bar{x}_{n+1}z_{n+1}=0.$$    Noting that $\p_2=(x_1, 1, x_3,\cdots,x_{n+1})^T$ and
	$ |x_1|^2+\sum_{i=3}^n|x_i|^2=|x_{n+1}|^2$, we have \begin{eqnarray*}
		|\bar{x}_{n+1}z_{n+1}|^2&=&|\bar{x}_1z_1+\bar{x}_3z_3+\cdots+\bar{x}_nz_n|^2\\
		&\leq & (|x_1|^2+\sum_{i=3}^n|x_i|^2)(|z_1|^2+\sum_{i=3}^n|z_i|^2).
	\end{eqnarray*}
	This implies that    $\p_1^{\perp}\cap \p_2^{\perp}\subset V_+\cup V_0$ and  $\p_1^{\perp}\cap \p_2^{\perp} \cap V_0=(\p_1-\p_2)\bh$.
	
	We consider the case  $t>1$.
	Noting that  $x_2=t>1$,
	we have $|x_{n+1}|^2-(|x_1|^2+\sum_{i=3}^n|x_i|^2)=|x_2|^2-1>0$.   Similarly  each  $\z\in \p_1^{\perp}\cap \p_2^{\perp}$ is of the form $=(z_1,0,z_3,\cdots,z_{n+1})^T$  satisfying the following equation  $$\bar{x}_1z_1+\bar{x}_3z_3+\cdots+\bar{x}_nz_n-\bar{x}_{n+1}z_{n+1}=0.$$    
	Direct computation shows  that $V\cap \p_1^{\perp}=\z\bh$,  where $\z=(x_1,0,x_3,\cdots,x_{n+1})^T\bh\in V_-$,  and  $V\cap \p_2^{\perp}=\w\bh$,  where $\w=(x_1, (|x_2|^2-1)\bar{x}_2^{-1},x_3,\cdots,x_{n+1})^T\bh\in V_-$.
	Hence  $$(\p_1^{\perp}\cap  V_-)\cap (\p_2^{\perp}\cap  V_-)=\emptyset.$$
	We mention that  $l_{\p_i}=\bp(\p_i^{\perp}\cap  V_-), i=1,2 $ are two totally geodesic submanifolds which  are equivalent to  ${\bf H}_{\bh}^{n-1}$.
		It follows from (\ref{disformula}) that  $$|\langle \p_1, \p_2 \rangle |=\cosh \left(\frac{\rho(\bp(\z),\bp(\w))}{2}\right)=|x_2|.$$
	Let $\z=(z_1,\cdots,z_{n+1})^T \in {\p_2}^{\perp}\cap V_-$  and,  for simplicity,
		denote  by $$X=\sqrt{\sum_{i\neq 2,n+1}|x_i|^2},\  Z=\sqrt{\sum_{i\neq 2,n+1}|z_i|^2}. $$
		Then 	$1-|x_2|^2=X^2-|x_{n+1}|^2$ and
	\begin{equation}\label{eqtt}
		|\bar{x_2}z_2|^2=|\bar{x}_{n+1}z_{n+1}-(\bar{x}_1z_1+\bar{x}_3z_3+\cdots+\bar{x}_nz_n)|^2
		\geq (|\bar{x}_{n+1}z_{n+1}|-XZ)^2.
	\end{equation}
	Let $$K=\cosh^2\frac{\rho(\bp({\p_1}^{\perp}\cap  V_-),\bp(\z))}{2}-|\langle{\p_1}, \p_2\rangle|^2.$$
	By Lemma \ref{distline} and  (\ref{eqtt}),  we obtain
	\begin{eqnarray*}		 K&=&1+\frac{|z_2|^2}{|z_{n+1}|^2-|z_1|^2-|z_3|^2-\cdots-|z_n|^2}-|x_2|^2\\&=&\frac{|z_2|^2+(X^2-|x_{n+1}|^2)(|z_{n+1}|^2-Z^2-|z_2|^2)}{|z_{n+1}|^2-Z^2-|z_2|^2}\\&=&\frac{|x_2|^2|z_2|^2+|z_{n+1}|^2X^2-   X^2Z^2-|x_{n+1}|^2|z_{n+1}|^2+ |x_{n+1}|^2Z^2  }{|z_{n+1}|^2-Z^2-|z_2|^2}\\
		&\geq &\frac{(|z_{n+1}|X-|x_{n+1}|Z)^2 }{|z_{n+1}|^2-Z^2-|z_2|^2}\geq 0.\end{eqnarray*}
	This  inequality implies that the real geodesic connecting $\z$ and $\w$ is the shortest curve form $l_{\p_1}$  to  $l_{\p_2}$.
\end{proof}

\subsection{Moduli space on $\bp(V_+)$ of case $m=2$}\label{sec4sub2}
We need the following fact, which is easy to verified.  We  refer to \cite{bisgen09,caopar04}  for more details of  ${\rm Sp}(1,1)$.
\begin{lem}\label{lem-embed}
	Let  $g\in \spt$  and ${\bf e}_2=(0,1,0)^T\in \bh^{2,1}$ such that  $g{\bf e}_2={\bf e}_2\mu.$  Then  $g$ is of the form
	$$g= \left(
	\begin{array}{ccc}
	a &0&b\\
	0 &\mu&0\\
	c &0&d
	\end{array}
	\right),
	$$
	where
	$$  \left(
	\begin{array}{cc}
	a &b\\
	c &d
	\end{array}
	\right)\in {\rm Sp}(1,1)\ \ \mbox{and} \ \ \mu\in {\rm Sp}(1). $$
	\end{lem}

\begin{thm}
	The  configuration space $\mathcal{M}(n,2)$ is homeomorphic to $\br^{\geq}=\{t\in \br: t\geq 0\}$.
\end{thm}
\begin{proof}  By Remark \ref{limdim},  we  can work in $\bh^{2,1}$ in this situation.  Noting   the normalization (\ref{twonorm}), we only need to show that there exists a $g\in \spt$ such that $g\p_1=\q_1\lambda_1$ and  $g\p_2=\q_2\lambda_2$ when $G((\p_1,\p_2))=G((\q_1,\q_2))=\left(\begin{array}{ccc}
	1 &t\\
	t &1\\
	\end{array}
	\right)$.  Noting  Proposition \ref{mutuorth}, we only need to consider the case  $t\neq 0$.  Observe that  $t\neq 0$ implies  $\lambda_1=\lambda_2$.
	Since $\spt$  acts transitively on  $\bp(V_+)$,   we may further assume that  $$\p_1=\q_1=(0,1,0)^T,\  \p_2=(x_1,t,x_3),\ \q_2=(y_1,t,y_3)^T,$$
	where $|x_3|^2-|x_1|^2=|y_3|^2-|y_1|^2=t^2-1$.   By Lemma \ref{lem-embed},   we need  to  find an element  $f=\left(
	\begin{array}{cc}
	a &b\\
	c &d
	\end{array}
	\right)\in {\rm Sp}(1,1)$  mapping $(x_1,x_3)^T$ to $(y_1,y_3)^T\mu$.  The fact that  ${\rm Sp}(1,1)$  acts doubly transitively on $\partial {\bf H}_{\bh}^1$,  transitively on ${\bf H}_{\bh}^1$, and  $\bp(V_+)$ respectively,  completes the proof.
\end{proof}

\section{The structure of Gram matrices of points  on  $\bp(V_+)$ }\label{sec5v+}
In this section, we provide a  $1$-normalized Gram matrix for an  $m$-tuple  on  $\bp(V_+)$.   The main purpose of this section is to  refine the structures of Gram matrices. These refined structures are crucial in introducing  new invariants in non regular case and the block-normalized algorithm in regular case.

\subsection{$1$-normalized Gram matrix}
\begin{prop}\label{normprocess1}
	Let  $\mathfrak{p}=(p_1,\cdots,p_m)$ be an $m$-tuple of pairwise distinct points in $\bp(V_+)$.  Then the equivalence class of Gram matrices associated to $\mathfrak{p}$ contains a matrix $G= (g_{ij})$ with
	$$g_{ii} = 1,\ i=1,\cdots, m,\   g_{1j}\geq 0,
	j=2,\cdots, m.$$
\end{prop}

\begin{proof}  Let $\p=(\p_{1},\cdots, \p_{m})$ be an arbitrary lift of $\mathfrak{p}$.   We want to obtain  a diagonal matrix $D_1$  such that  $G(\p D_1)$  is the desired Gram matrix.
	
	We may assume that  $\langle \p_{i}, \p_{i}\rangle=1$ by noticing that   $$\langle \p_{i}\lambda_i, \p_{i}\lambda_i\rangle=1,  \ \mbox{for}\  \lambda_i=\sqrt{\frac{1}{\langle \p_{i}, \p_{i}\rangle} }.$$
	For $i=2,\cdots,m$, let
	\begin{equation}\label{findli2}\lambda_i=\left\{
	\begin{array}{ll}
	\frac{\langle {\p}_{1},{\p}_{i} \rangle}{|\langle {\p}_{1},{\p}_{i}\rangle|}, & \hbox{provided}\   \langle{\p}_{1},{\p}_{i}\rangle\neq 0;\\
	1, &    \hbox{otherwise}.
	\end{array}
	\right.	\end{equation}
	Then   there exists a  $\lambda_1\in {\rm Sp}(1)$  such that  $\bar{\lambda_1}\bar{\lambda_3}\langle \p_{2},\p_{3}\rangle \lambda_2\lambda_1$ is a complex number with no-negative imaginary part if $\langle \p_{2},\p_{3}\rangle\neq 0$.  Then  $G(\p_{1}\lambda_1, \p_{2}\lambda_2\lambda_1,\cdots, \p_{m}\lambda_m\lambda_1)$ is the desired Gram matrix.
	In other words, $G(\p D_1)$  is the desired Gram matrix with
\begin{equation}\label{d-1} D_1={\rm diag}\Big(\sqrt{\frac{1}{\langle \p_{1}, \p_{1}\rangle} }\lambda_1, \sqrt{\frac{1}{\langle \p_{2}, \p_{2}\rangle} }\lambda_2\lambda_1,\cdots, \sqrt{\frac{1}{\langle \p_{m}, \p_{m}\rangle} }\lambda_m\lambda_1\Big).\end{equation}
\end{proof}

\begin{dfn}
	The Gram matrix  $G$ as in Proposition \ref{normprocess1} of the form
	\begin{equation}\label{semimatrixp}
	G=(g_{ij})=\left(
	\begin{array}{cccccc}
	1 & g_{12} & g_{13}  & g_{14}  & \cdots & g_{1m}  \\
	g_{12}& 1 & g_{23} & g_{24}&\cdots & g_{2m} \\
	g_{13} & \overline{g_{23}} & 1 & g_{34}&\cdots & g_{3m}  \\
	g_{14} & \overline{g_{24}} & \overline{g_{34}} & 1&\cdots & g_{4m}\\
	\vdots &  \vdots  &  \vdots  &  \vdots &\ddots &  \vdots \\
	g_{1m} & \overline{g_{2m}} & \overline{g_{3m}} & \overline{g_{4m}}&\cdots & 1\\
	\end{array}
	\right)
	\end{equation}
	is called  the {\it  1-normalized Gram matrix}.
\end{dfn}

The following result can be shown similarly  as  Proposition \ref{prop-condgn}. \begin{thm}(\cite[Propsition 3.2]{cungus12fr})\label{thmcondgp}
	Let $G= (g_{ij})$ be  a  Hermitian $m\times m$-matrix, $m>2$  with
	$$g_{ii} = 1,\ i=1,\cdots, m,\   g_{1j}\geq 0,
	j=2,\cdots, m.$$  Let $i(G)=(n_+,n_-,n_0)$.
	Then $G$ is a  1-normalized Gram matrix  associated with  an $m$-tuple	of pairwise  distinct  points in  $\bp(V_+)$   if and  only if
	\begin{equation}\label{inercondpg} 1\leq n_++n_-\leq n+1,n_+\le  n, \ n_-\le 1 , \  n_++n_-+n_0=m.\end{equation}
\end{thm}

\begin{rem}
 The number $1$ of 1-normalized Gram matrix is  equivoke with meaning  that  we  normalize the Gram matrix in the view point  standing in our ordered  position $1$, as well as  with the meaning  that  we normalize the points in $\bp(V_+)$ with the properties $\langle \p_{i}, \p_{i}\rangle=1$.    It specifies the entries in row 1 (together column 1)  and leaves  entries in  other rows  ambiguity (even in the complex case).  This  phenomenon motivates the development of  block-normalized algorithm.  By the content in Section \ref{sec2.3ine}, we can state similar Theorem \ref{thmcondgp}  for other  normalized form of Gram matrix because of the invariability of  (\ref{inercondpg}). So we can focus on constructing of the parameter space in the sequence.
\end{rem}

\subsection{The structure of Gram matrices of points  on  $\bp(V_+)$ }
In what follows, we assume that $G(\p)$ is already  a 1-normalized Gram matrix.
The following proposition may be regarded as a generalization of  Theorem \ref{thm-twpo} (ii)
\begin{prop}\label{polar-inter-n}
	Let  $\p=(\p_1,\cdots,\p_t)$ be a $t$-tuple of pairwise distinct points in $\bp(V_+)$  satisfying   $$\langle {\p}_i, {\p}_j\rangle=1,\ i,j=1,\cdots t$$ and  $V={\rm span}\{\p_1,\cdots,\p_l\}$.
	Then there exists a unique  fibre $\z\bh\in V_0$ such that  $$V\subset \z^{\perp},  V\cap V_0=\z\bh, V\cap V_-=\emptyset.  $$
	In fact  $$\z=\p_2-\p_1, V={\rm span}\{\p_1,\p_2\}={\rm span}\{\z,\p_1\}.$$
\end{prop}

\begin{proof}
	Let $\bu=\p_1\lambda_1+\p_2\lambda_2\in V$.
	Then \begin{equation}\langle \bu, \bu \rangle=|\lambda_1|^2+|\lambda_2|^2+2\Re(\bar{\lambda}_2\lambda_1)\geq 0.\end{equation}
	Note that $\langle \bu, \bu \rangle=0$ if and only if $\lambda_1=-\lambda_2$.   Hence  $(\p_2-\p_1)\bh$ is the unique fibre of the intersection  ${\rm span}\{\p_1,\p_2\} \cap V_0$  and ${\rm span}\{\p_1,\p_2\}\cap V_-=\emptyset$.  Noting that $\langle \p_i-\p_j, \p_i-\p_j\rangle=0$ and   $\langle  \p_i-\p_j,  \p_2-\p_1)\rangle=0$ for $i\neq j$,  by Proposition \ref{prop-1.1}  we have $(\p_2-\p_1)\bh=(\p_i-\p_j)\bh$.  	  Since  $\langle \p_i, (\p_1-\p_2) \rangle=0, i=1,\cdots, t$, we have   $V\subset (\p_2-\p_1)^{\perp}$.   It follows from  $\p_i-\p_1\in(\p_2-\p_1)\bh$ that  there exist  $\lambda_i $  such that   $$\p_i=\p_1+(\p_2-\p_1)\lambda_i=\p_2\lambda_i+\p_1(1-\lambda_i),i=1,\cdots,t.$$   This implies that $$V={\rm span}\{\p_1,\p_2\}={\rm span}\{z,\p_1\}$$ and therefore
	$V\cap V_0=\z\bh, V\cap V_-=\emptyset. $
\end{proof}
 The information of $\lambda_i$  disappears in the  sub Gram matrix $G((\p_1,\cdots,\p_t))$.  Moreover, such information  can not be rebuilt  through  the relationships  with other points in some situations.    This implies that the Gram matrix loses the configuration information of such a $t$-tuple.
We provide the following explicit example in ball model  to  illustrate  this phenomenon. We remind that Cunha et al provided  a proof of similar example involving the fixed point theory of complex hyperbolic isometries  in \cite[Section 5]{cungus12fr}.
\begin{example}\label{lostinf}
	Let $\z=(1,0,1)^T\in V_0$ and $\p_1=(0,1,0)\in V+$.
	Let $\p_i=\p_1+i\z,i=2,3$.  Then   $$G((\p_1,\p_2,\p_3))=G((\p_3,\p_2,\p_1))=\left(
	\begin{array}{ccc}
	1 &1&1\\
	1 &1&1\\
	1 &1&1
	\end{array}
	\right).
	$$
		We claim that  $(\bp(\p_1),\bp(\p_2),\bp(\p_3))$ and $(\bp(\p_3),\bp(\p_2),\bp(\p_1))$  are not $\pspt$-congruent.
\end{example}
	\begin{proof}[Proof of the Claim]  Suppose that the two triples  above are   $\pspt$-congruent. Then there exist  a $g\in \spt$ such that $$g\p_1=\p_3\lambda_1, g\p_2=\p_2\lambda_2, g\p_3=\p_1\lambda_3.$$   It follows from $$\langle g\p_i, g\p_i\rangle=\langle g\p_i, g\p_j\rangle=\langle \p_i, \p_i\rangle=1$$ that  $\lambda_i\in {\rm Sp}(1)$ and $\bar{\lambda_j}\lambda_i=1$, and therefore  $\lambda_1=\lambda_2=\lambda_3:=\lambda$.
		Hence  $$g2\z=g(\p_2-\p_1)=(\p_2-\p_3)\lambda=-\z \lambda,$$ which contradicts $$g\z=g(\p_3-\p_2)=(\p_1-\p_2)\lambda=-\z 2\lambda.$$
	\end{proof}

If $V$ is parabolic,  by Proposition \ref{polar-inter-n}  we can refine Theorem \ref{thm-inertia} as follows.
\begin{prop}\label{deg-case}
	Let $\p=(\p_1,\cdots, \p_m)\in  \bh_{n+1,m}$, $V={\rm span}\{\p_1,\cdots, \p_m\}$ and
	$$\dim_q V=k+1,\ i(G(\p))=i(\p^*J\p)=(k,0,m-k).$$
	Then $S(m)=\{1,\cdots,m\}$ has a partition:
\begin{equation}S_i=\{s_{i1},\cdots,s_{it_{i}}\},s_{i1}<\cdots<s_{it_{i}},i=1,\cdots,k\end{equation}
	with the properties  \begin{equation}\label{pstrcture} S(m)=\bigcup_{i=1}^k S_i; \langle \p_{s_{il}}, \p_{s_{id}}\rangle=1,1\leq l,d\leq t_i; \langle\p_{s_{il}}, \p_{s_{jd}}\rangle=0,i\neq j\end{equation}
	and in each $$\p_{S_i}:=(\p_{s_{i1}},\cdots,\p_{s_{it_i}})$$ we can not partition likewise as in (\ref{pstrcture}).

	There exists a common  $\z_0\in V_0$ such that
	$\p\in \z_0^{\perp}$ and \begin{equation}\label{vicoordinates}\p_{s_{il}}=\p_{s_{i1}}+\z_0 \lambda_{il}, 1<l\le  {\rm Card}(S_i), i=1,\cdots,k,\end{equation}
	where ${\rm Card}(S_i)$ is the cardinality of $S_i$. We define \begin{equation}\label{spacevi}V_i={\rm span}\{\p_{s_{i1}},\cdots, \p_{s_{it_i}}\}={\rm span}\{\p_{s_{i1}},\z_0\}, i=1,\cdots,k.\end{equation}
\end{prop}

\medskip

If $V$ is not parabolic, we can refine Theorem \ref{thm-inertia} as follows.
\begin{prop}\label{regula-case}
	Let $\p=(\p_1,\cdots, \p_m)\in  \bh_{n+1,m}$, $V={\rm span}\{\p_1,\cdots, \p_m\}$ and
	$$\dim_q V=k+1,\ i(G(\p))=i(\p^*J\p)=(k,1,m-k-1)\; \mbox{or}\ (k+1,0,m-k-1).$$
	Then $S(m)=\{1,\cdots,m\}$ has a partition:
\begin{equation}\label{partregu}S_i=\{s_{i1},\cdots,s_{it_{i}}\},s_{i1}<\cdots<s_{it_{i}} ,i=1,\cdots,s \end{equation}
	with the properties
	\begin{equation} S(m)=\bigcup_{i=1}^s S_i;\langle \p_{s_{il}}, \p_{s_{jd}}\rangle=0,i\neq j \end{equation}
	and in each $\p_{S_i}:=(\p_{s_{i1}},\cdots,\p_{s_{it_i}})$ we can not partition likewise as above.
\end{prop}
It is  helpful to keep in mind that there are no relationships  among  the  blocked-entries  corresponding  to   each components  $\p_{S_i}$  in the diagonal matrix  $D$ in (\ref{isolift1}). This is the motivation of refinement of  Theorem  \ref{thm-inertia}.  Furthermore, when  $V$  is  not parabolic, we still need to  partition the components $S_i$  in some situations.

\section{Moduli space on $\bp(V_+)$ of case $m\geq 3$: non regular cases}\label{sec6ireg}

We will work on the Siegel domain in this section. We will  construct
invariants which describe the $\pspn$-congruence classes when $V$  is  parabolic.

\medskip

We first recall the  following fact of isometries in $\spn$ fixing $\infty$.
\begin{lem}(c.f.\cite[Lemma 3.3.1]{chegre74})\label{ginf}
	Let $\z_{\infty}=(1, 0, \,\cdots,\,0,0)^T$,  $\bp(\z_{\infty})=\infty$ and 	$$G_{\infty}=\{g\in \spn:
	g(\infty)=\infty\}.$$  Then   $g\in G_{\infty}$  is of the form
	\begin{equation}\label{gw}g=\left(
	\begin{array}{ccc}
	\lambda & \gamma^* & s\\
	0 & U & \beta \\
	0& 0 &\mu \\
	\end{array}
	\right), \end{equation}where $ \lambda,\mu,s\in \bh,\beta,\gamma\in \bh^{n-1}, U\in {\rm Sp}(n-1),  |\bar{\mu}\lambda|=1,\
	\;\Re(\bar{\mu}s)=-\frac{1}{2}\left|\beta\right|^2,\;\beta=-U\gamma \mu.$
\end{lem}

Let $\p=(\p_{1},\cdots, \p_{m})$ and $\q=(\q_{1},\cdots, \q_{m})$  be two ordered  $m$-tuples of pairwise  distinct
  points in  $\bp(V_+)$ such that $V(\p)$ and $V(\q)$ are parabolic.      Observe that  if  $\p$ and   $\q$  are  $\pspn$-congruent then
   they have the same structure  given by Proposition \ref{deg-case}. Since $\spn$  acts doubly transitively on  $\bhq$, we can further assume that $\p,\q\in \z_{\infty}^{\perp}$.    As showed  by  Example \ref{lostinf}, besides  the information of   structure,  other conditions  are needed  for $\p,\q$  being   $\pspn$-congruent.

\medskip

   In what follows, we assume that $m\geq 3$,   $V(\p)={\rm span}\{\p_1,\cdots, \p_m\}$ is parabolic and $V(\p) \subset \z_{\infty}^{\perp}$.
It is obvious that $$\z_{\infty}^{\perp}=(z_1,\cdots,z_n,0)^T:=(z_1,\alpha^T,0)^T.$$  Therefore the action of $g\in G_{\infty}$ on $\z_{\infty}^{\perp}$ can be expressed by
$$g:\left(
\begin{array}{c}
z_1\\
\alpha\\
0\\
\end{array}
\right)\to \left(
\begin{array}{c}
\lambda z_1+\gamma^*\alpha\\
U\alpha\\
0\\
\end{array}
\right).$$
The restriction of the Hermitian form $\langle ,\rangle$ on $\z_{\infty}^{\perp} $ is the usual inner product on $\bh^{n-1}$, i.e.,
$$\langle (k_1,\alpha_1,0)^T , (k_2,\alpha_2,0)^T\rangle=\alpha_2^*\alpha_1.$$
For  $g$ of the form (\ref{gw}), we define the map \begin{equation}\label{dfnpi}  \Pi:g\in G_{\infty}\to \tilde{g}=\left(
\begin{array}{cc}
\lambda & \gamma^* \\
0 & U  \\
\end{array}
\right)\in \tilde{G}_{\infty}.\end{equation}  Then  $\Pi$ is a homomorphism  with
$$\ker(\Pi)=\left\{\left(
	\begin{array}{ccc}
	1 & 0 & s\\
	0 & I_{n-1} & 0 \\
	0& 0 & 1 \\
	\end{array}
	\right)\ \mbox{with} \ \Re(s)=0\right\}$$
and its homomorphic image $\tilde{G}_{\infty}=\Pi(G_{\infty})$ is a subgroup    of   ${\rm GL}(n,\bh)$.   The action of  $G_{\infty}$ on $\z_{\infty}^{\perp} $ can be  expressed by the projection action of   $\tilde{G}_{\infty}$  on $\bh \bp^{n-1}=(z_1,\alpha^T)^T$.

Noting  Proposition \ref{deg-case} and  $G(\p)$  being  a 1-normalized Gram matrix, we have
	\begin{equation}\label{cooralph}\tilde{\p}_{s_{il}}=(k_{il},\alpha_i^T)^T,\ \mbox{for}\  s_{il}\in S_i \end{equation}
and  $$\alpha_i^*\alpha_i=1,  1\le i \le k.$$ Therefore  there exists a  $U\in   {\rm Sp}(n-1) $ such that $g={\rm diag}(1,U,1)\in  {\rm Sp}(n,1)$ satisfying \begin{equation}\label{univect}U(\alpha_1,\cdots, \alpha_k)=(\e_1,\cdots,\e_k),\end{equation}
where $\e_i, 1\le i \le k$ are  $k$ vectors in the standard basis of $\bh^{n-1}$.  Therefore we may further reformulate (\ref{cooralph}) as
	\begin{equation}\label{coorbasis}\tilde{\p}_{s_{il}}=(k_{il},e_i^T)^T,\ \mbox{for}\  s_{il}\in S_i.\end{equation}
In order to parameterize the moduli space, we  introduce the following map $\phi$  to give the corresponding  coordinates in $\bh\cup \infty$ for vectors in $V_i={\rm span}\{\p_{s_{i1}},\z_{\infty}\}$:
\begin{equation}\label{coordinates}\phi(\z_{\infty})=\infty, \phi(\tilde{\p}_{s_{il}})=k_{il}, 1\le l \le t_i; 1\le i \le k.\end{equation}

Let $h=\left(
\begin{array}{cc}
	\lambda & \gamma^* \\
	0 & U  \\
\end{array}
\right)$ with $U(\e_1,\cdots,\e_k)=(\e_1,\cdots,\e_k)$ and $\gamma=(c_1,\cdots,c_{n-1})^T$.
Note that\begin{equation}\label{transcor}\left(
\begin{array}{cc}
\lambda & \gamma^* \\
0 & U  \\
\end{array}
\right)\left(
\begin{array}{c}
k_{il}\\
\e_i\\
\end{array}
\right)=\left(
\begin{array}{c}
\lambda k_{il}+ \gamma^*\e_i\\
\e_i\\
\end{array}
\right)=\left(
\begin{array}{c}
\lambda k_{i1}+ c_i\\
\e_i\\
\end{array}
\right).\end{equation}
This means  the restriction of $h$ in $V_i$ is  \begin{equation}\label{hipara} h_i: k_{il}\to  \lambda k_{i1}+ c_i, 1\le i \le k.\end{equation}
The above treatment  can be thought of as  introducing  the inhomogeneous coordinates in each $V_i$.  Form this point of view, the restriction  of  an element $\tilde{g}$ of  form  (\ref{dfnpi}) to $V_i$ is a quaternionic M\"obius transformation in $\Gamma_{\infty}$, the isotropy group at $\infty$  in ${\rm PS}_\triangle L(2,\bh)$  \cite{caopems07}.

Summarizing the above descriptions, we  have so far  defined a map
 \begin{equation}\label{dfnpii}  \Pi_i:g\in G_{\infty}\to h_i=\left(
\begin{array}{cc}
\lambda & c_i \\
0 & 1  \\
\end{array}
\right)\in \Gamma_{\infty}\end{equation}
and the action of $g$ on $\z_{\infty}^{\perp}$ is inherited by the actions of $h_i$ on $V_i$, which is identified with  $\overline{\bh}$.

Observe that the coordinates defined by  (\ref{coordinates}) contain the information of $\lambda_{il}$ in (\ref{vicoordinates}).  To distinguish between $\pspn$-congruence classes of $m$-tuples in degenerate case is the same as  distinguishing the $h_i$-congruence classes in  $V_i$  for all $i$.  For this purpose, we need to introduce new geometric invariants which are  invariant under the  action of $h_i$.

\begin{dfn}(\cite[Definition 4.2]{bisgen09})
The quaternionic cross-ratio of four points $z_1, z_2, z_3, z_4\in \bh\cup \infty$ is defined as
$$[z_1, z_2, z_3, z_4]=(z_1-z_3)(z_1-z_4)^{-1}(z_2-z_4)(z_2-z_3)^{-1}.$$
\end{dfn}

\begin{lem}(\cite[Proposition 4.1]{bisgen09})\label{lem-10infty}
Given three distinct $z_1,z_2,z_3\in \bh$, the element $f\in {\rm PS}_\triangle L(2,\bh)$  defined by
\begin{equation}
f(z)=(z_3-z_2)(z_3-z_1)^{-1}(z-z_1)(z-z_2)^{-1}
\end{equation}
maps $z_1$ to $0,z_2$ to  $\infty$ and $z_3$
 to $1$. Moreover, all elements $f\in {\rm PS}_\triangle L(2,\bh)$  with
the same property are of the form:
$$\lambda I_2\circ f(z)=\lambda f(z)\lambda^{-1}$$
with $\lambda\in \bh-\{0\}$.
\end{lem}
It follows from \cite{caopems07} that  an element $f\in {\rm PS}_\triangle L(2,\bh)$ fixing   $0,1,\infty$ is of the form $f=\left(
\begin{array}{cc}
\lambda & 0 \\
0 & \lambda  \\
\end{array}
\right)=\lambda I_2.$
Based on this observation and  \cite[Proposition 4.4]{bisgen09},  the cross-ratios enjoy the following properties.

\begin{lem}\label{lem-croproperty}
\begin{itemize}
	\item[(1)]   For any $z\in \bh$  such that $z \neq  0$ and $z\neq 1, [z,1, 0,\infty] =z$;
	\item [(2)]  Given distinct points $z_1,z_2,z_3,z\in  \overline{\bh}$,   $$[f(z),f(z_3), f(z_2), f(z_1)]=\lambda_f [z,z_3, z_2, z_1]\lambda_f^{-1},$$
	where  $\lambda_f$ is a  quaternion solely depending on $f\in {\rm PS}_\triangle L(2,\bh)$. In particular, for $h_i$ given by (\ref{hipara}),  we have $$[h_i(z),h_i(z_3), h_i(z_2), h_i(z_1)]=\lambda [z,z_3, z_2, z_1]\lambda^{-1}.$$
	\end{itemize}
\end{lem}

\begin{dfn}
	We introduce the following geometric invariants in component   $\p_{S_i}=(\p_{s_{i1}},\cdots,\p_{s_{it_i}})$ when  ${\rm Card}(S_i)\geq 3$:
	\begin{equation}\label{chidfn}\chi(\tilde{\p}_{s_{il}},\tilde{\p}_{s_{ij}},\tilde{\p}_{s_{it}})=(k_{il}-k_{it})(k_{ij}-k_{it})^{-1}
=[k_{il},k_{ij},k_{it},\infty].\end{equation}
\end{dfn}
We mention that since the points $\tilde{\p}_{s_{i1}},\tilde{\p}_{s_{ij}},\tilde{\p}_{s_{it}}$ are all distinct, $\chi$ is finite and $\chi\neq 0,1$. Therefore $$\chi\in\bh-\{0, 1\}.$$

  To sort out the conditions for  $\p$ and   $\q$  being  $\pspn$-congruent,  w.o.l.g,  we may  assume that   $\p,\q\in \z_{\infty}^{\perp}$ have the same structure  given by Proposition \ref{deg-case}.   We denote the corresponding coordinates of $\tilde{\q}_{S_i}$ by  \begin{equation}\label{qcoorbasis}\tilde{\q}_{s_{il}}=(w_{il},e_i^T)^T,\ \mbox{for}\  s_{il}\in S_i\end{equation}
 and compute  the corresponding invariants of  $\q$ in the same manner  as these of  $\p$.

\medskip

We first obtain the necessary and sufficient condition of two triples  being  $\pspn$-congruent directly.
\begin{prop}\label{threep}
 $\tilde{\p}_i=(\tilde{\p}_{s_{i1}},\tilde{\p}_{s_{i2}},\tilde{\p}_{s_{i3}})$ and  $\tilde{\q}_i=(\tilde{\q}_{s_{i1}},\tilde{\q}_{s_{i2}},\tilde{\q}_{s_{i3}})$ are $\pspn$-congruent  if and only if there exists a $\lambda \in \bh-\{0\}$ such that $$\chi(\tilde{\p}_{s_{i1}},\tilde{\p}_{s_{i2}},\tilde{\p}_{s_{i3}})=\lambda \chi(\tilde{\q}_{s_{i1}},\tilde{\q}_{s_{i2}},\tilde{\q}_{s_{i3}})\lambda^{-1}.$$
\end{prop}

\begin{proof}
	Assume that  $\tilde{\p}_i$ and  $\tilde{\q}_i$ are $\pspn$-congruent. Let  $\p_i$ and $\q_i$ be the corresponding triples in $\bh^{n,1}$.
 Then there exists a  $g\in G_{\infty}$ such that $g({\p}_{s_{i1}},{\p}_{s_{i2}},{\p}_{s_{i3}})=({\q}_{s_{i1}}\nu_1,{\q}_{s_{i2}}\nu_2,{\q}_{s_{i3}}\nu_3)$.  As before, we know that $\nu_1=\nu_2=\nu_3:=\nu$ and $|\nu|=1$.  This implies that $ \tilde{g}\tilde{\p}_i=\tilde{\q}_i\nu$, i.e.,
	$$ \lambda k_{il}+\gamma^*\e_i=w_{il}\;\nu,U\e_i=\e_i\nu,l=1,2,3.$$
	Therefore $$\lambda( k_{i2}- k_{i3})=(w_{i2}- w_{i3})\nu,\lambda( k_{i1}- k_{i3})=(w_{i1}- w_{i3})\nu.$$
	Hence  $$\lambda(k_{i1}-k_{i3})(k_{i2}-k_{i3})^{-1}\lambda^{-1}=(w_{i1}-w_{i3})(w_{i2}-w_{i3})^{-1}.$$

	Conversely,  suppose that $\chi(\tilde{\p}_{s_{i1}},\tilde{\p}_{s_{i2}},\tilde{\p}_{s_{i3}})\sim\chi(\tilde{\q}_{s_{i1}},\tilde{\q}_{s_{i2}},\tilde{\q}_{s_{i3}})$.   Then there exists a $\lambda$ such that $$\lambda (k_{i1}-k_{i3})(k_{i2}-k_{i3})^{-1}\lambda^{-1}=(w_{i1}-w_{i3})(w_{i2}-w_{i3})^{-1}.$$
	We may further require that $|\lambda|=\frac{|k_{i1}-k_{i3}|}{|w_{i1}-w_{i3}|}$.  Let $\nu=(w_{i1}-w_{i3})^{-1}\lambda (k_{i1}-k_{i3})$. Then  $$\lambda (k_{i1}-k_{i3})=(w_{i1}-w_{i3})\nu, \lambda (k_{i2}-k_{i3})=(w_{i2}-w_{i3})\nu.$$
	From the above two equalities, we have $\lambda (k_{i1}-k_{i2})=(w_{i1}-w_{i2})\nu$.  We can find  a $\gamma\in \bh^{n-1}$  and a $U\in {\rm Sp}(n-1)$ satisfying  $$\lambda k_{i1}+\gamma^*\e_i=w_{i1}\;\nu, U\e_i=\e_i \nu.$$
	The above equalities  also imply  $$\lambda k_{i2}+\gamma^*\e_i=w_{i2}\;\nu,\lambda k_{i3}+\gamma^*\e_i=w_{i3}\;\nu.$$
	With $\lambda,\gamma, U$ above,  we can construct a $g\in G_{\infty}$ of the form (\ref{gw}) satisfying $$g({\p}_{s_{i1}},{\p}_{s_{i2}},{\p}_{s_{i3}})=({\q}_{s_{i1}}\nu_1,{\q}_{s_{i2}}\nu_2,{\q}_{s_{i3}}\nu_3).$$
\end{proof}

 Translating  Example \ref{lostinf} from  ball model to Siegel domain model, one has an instance of positive points: $$\p_1=(0,1, 0)^T,\p_2=(2\sqrt{2},1,0)^T, \mbox{and}\ \p_3=(3\sqrt{2},1,0)^T.$$  Observe that $\chi(\p_1,\p_2,\p_3)=3$ and $\chi(\p_3,\p_2,\p_1)=3/2$. Therefore $(\bp(\p_1),\bp(\p_2),\bp(\p_3))$ and $(\bp(\p_3),\bp(\p_2),\bp(\p_1))$  are not $\pspt$-congruent. For the case of  more than three points, it is convenient to use the quaternionic cross-ratios.

\begin{prop}\label{mpoints} Let $\p = (z_1,\cdots, z_m)$ and $\q
= (w_1,\cdots, w_m)$ be two ordered $m$-tuples
of pairwise distinct points in $\overline{\bh}$, $m \geq 4$. Then  $\p$ and $\q$
are congruent with respect to the
diagonal action of  $ {\rm PS}_\triangle L(2,\bh)$  if and only if there exists a $\lambda\in \bh-\{0\}$ such that
\begin{equation}\label{planecondition6} [z_j,z_3, z_2, z_1] = \lambda [w_j,w_3,w_2,w_1]\lambda^{-1}, \forall \, \, \,   4\le j\le m.\end{equation}
\end{prop}
\begin{proof}
	If there is an  $f\in {\rm PS}_\triangle L(2,\bh)$ such that $f(z_j)=w_j, j=1,\cdots,m$.  Then by Lemma \ref{lem-croproperty}, the conditions of  (\ref{planecondition6}) hold.
	
	Conversely,
 assume that \begin{equation*}\label{planecondition} [z_j,z_3, z_2, z_1] = \lambda [w_j,w_3,w_2,w_1]\lambda^{-1}, \forall   4\le j\le m.\end{equation*}   By Lemma \ref{lem-10infty}  we
can find $f, g\in  {\rm PS}_\triangle L(2,\bh)$  such that  $f (z_3) = 1, f (z_2) = 0,f (z_1)=\infty$, and $g (w_3) = 1, g (w_2) = 0,g (w_1)=\infty$.  It follows from Lemma \ref{lem-croproperty} that
$$f(z_j) = [f(z_j),1,0,\infty] = [f(z_j), f (z_3), f (z_2), f (z_1)] = \lambda_f [z,z_3, z_2, z_1]\lambda_f^{-1}$$
and
$$g(w_j) = [g(w_j),1,0,\infty] = [g(w_j), g (w_3), g (w_2), g (w_1)] = \lambda_g [w,w_3, w_2, w_1]\lambda_g^{-1}.$$
Therefore our assumption implies that  $$h(z_j)=g^{-1}\circ \lambda_g(\lambda_f\lambda)^{-1}I_2\circ f (z_i) = w_i,  i = 1,\cdots,m.$$  Hence $\p$ and $\q$ are ${\rm PS}_\triangle L(2,\bh)$-congruent.
\end{proof}

\begin{dfn}\label{sec6defpara}
For $S_i=\{s_{i1},\cdots,s_{it_{i}}\},i=1,\cdots,k$ with ${\rm Card}(S_i)> 3$, we  associate with  $S_i$ the following geometric  invariants:
$$\chi_{i0} = \chi(\p_{s_{il}}, \p_{s_{i2}}, \p_{s_{i3}}), \chi_{i1} = \chi(\p_{s_{il}}, \p_{s_{i2}}, \p_{s_{i4}}),\cdots,\chi_{i(t_i-3)} = \chi(\p_{s_{il}}, \p_{s_{i2}}, \p_{s_{it_i}}) $$
and $$X_i(\p)=(\chi_{i0},\cdots,\chi_{i(t_i-3)}).$$
Let  $X(\p)$ be the vector  whose components consisting of  $X_i(\p)$  above.
\end{dfn}

Taking  $z_1=w_1=\infty$ in  Proposition \ref{mpoints}, we get the following proposition.
\begin{prop}\label{Simpoints}
	Let $\p_{S_i}$ and  $\q_{S_i}$ belong to $\z_{\infty}^{\perp}$ with the same Gram matrix whose entries are  all equal to $1$. Then $\p_{S_i}$ and  $\q_{S_i}$  are $\pspn$-congruent  if and only if  there exists a $\lambda\in \bh-\{0\}$ such that $$X_i(\p)=\lambda X_i(\q)\lambda^{-1}.$$
\end{prop}

 We still need to generalize the above result to  the case  of $G(\p)$  and $G(\q)$ having stratum  structure.
\begin{prop}\label{wholepoints}
	  Let $\p=(\p_{1},\cdots, \p_{m})$ and $\q=(\q_{1},\cdots, \q_{m})$  be two $m$-tuples of pairwise distinct	positive points of non regular case. We also assume that  $\p$ and $\q$ have  the same structure  given by Proposition \ref{deg-case} with the property ${\rm Card}(S_i)\geq 3$ for some $i$. Then  $\p$ and $\q$ are  $\pspn$-congruent  if only if   there exists a $\lambda\in \bh-\{0\}$ such that $$X(\p)=\lambda X(\q)\lambda^{-1}.$$
	\end{prop}
\begin{proof}
	W.o.l.g, we assume that $\p,\q\in \z_{\infty}^{\perp}$.   If  there is an  $f\in\pspn$ such that $f(\p_i)=\q_i, j=1,\cdots,m$.  Then $f\in G_{\infty}$  is of the form \begin{equation}\label{gw1}f=\left(
	\begin{array}{ccc}
	\lambda & \gamma^* & \star\\
	0 & U & \star \\
	0& 0 &\star \\
	\end{array}
	\right) \end{equation}
	and  $\p,\q$ must have the same structure given by Proposition \ref{deg-case}. Here and in what follows, $\star$ stands for   an arbitrary entry satisfying constraint that the corresponding matrix $f$  belongs  to  $\spn$.   By our normalization, we have
	\begin{equation}\label{univect1}U(\e_1,\cdots,\e_k)=(\e_1,\cdots,\e_k)\end{equation}
	and in each block  of   index  $S_i$,  we  also have
 \begin{equation}\label{gwc}\lambda k_{il}+\gamma^*\e_i=w_{il},1\le l\le  {\rm Card}(S_i).\end{equation}
Therefore we have $X(\p)=\lambda X(\q)\lambda^{-1}.$

	Conversely,   suppose that  $X(\p)=\lambda X(\q)\lambda^{-1}.$    By  Proposition \ref{Simpoints},  for  two specific blocks  $\p_{S_i}$ and  $\q_{S_i}$we can construct an element $f_i\in \spn$ of the form
	\begin{equation*}f_i=\left(
	\begin{array}{ccc}
	\lambda_i & \gamma_i^* & \star\\
	0 & U_i & \star \\
	0& 0 &\star \\
	\end{array}
	\right) \end{equation*} such that
	 \begin{equation}\label{gw2}U_ie_i=e_i,\lambda_i k_{il}+\gamma_i^*\e_i=w_{il},1\le l\le  {\rm Card}(S_i).\end{equation}
It is  a pleasant surprise  that  we can adjust $f_i$   to a suitable transformation which works for  $\p$ wholly  as follows.
First, it follows from Lemma \ref{lem-croproperty} that $\lambda_i=\lambda$.
Let $U\in {\rm Sp}(n-1)$ having the property $U(\e_1,\cdots,\e_k)=(\e_1,\cdots,\e_k)$.
It is obvious that  \begin{equation}\label{gw3}h_i=\left(
\begin{array}{ccc}
\lambda & \gamma_i^* & \star\\
0 & U & \star \\
0& 0 &\star \\
\end{array}
\right) \end{equation}
also maps $\p_{S_i}$ to  $\q_{S_i}$.
Note that $k\le n-1$.  Let  $$\gamma=(\gamma_1^*\e_1,\cdots, \gamma_k^*\e_k,\star,\cdots,\star)^T$$
and \begin{equation}\label{gw4}h=\left(
\begin{array}{ccc}
\lambda & \gamma^* & \star\\
0 & U & \star \\
0& 0 &\star \\
\end{array}
\right) .\end{equation}
Then one has the equations (\ref{gwc}),  and therefore  $\p$ and $\q$ are  congruent up to $h$.
	\end{proof}

By the above proof and   Section \ref{sec4sub2}, we have the following result which means that  structures of  Gram matrices  determine their congruent classes when ${\rm Card}(S_i)\le 2$ for all $i$.
\begin{prop}\label{wholepoints2}
	 Let $\p=(\p_{1},\cdots, \p_{m})$ and $\q=(\q_{1},\cdots, \q_{m})$  be two $m$-tuples of pairwise distinct	 positive points of non regular case with the same structure  given by Proposition \ref{deg-case} and ${\rm Card}(S_i)\le 2, i=1,\cdots,k$. Then  $\p$ and $\q$ are  $\pspn$-congruent.
\end{prop}

In order to describe the parameter space, we need the following result.
\begin{prop}\label{welldef}
	The  coordinates of  ${\bf O}_{X(\p)}$ given by rotation-normalized algorithm is well defined.
\end{prop}

\begin{proof}	
If  both $h_1,h_2\in \pspn$  map  $V$ to  a subspace of  $\z_{\infty}^{\perp}$. Then the coordinates in  (\ref{coordinates}) may be different from each other, which implies that $X(\p)$ in Definition \ref{sec6defpara} is dependent on the map $\phi$ in  (\ref{coordinates}).  However, since  $h_1^{-1}h_2\in G_{\infty}$, Lemma \ref{lem-croproperty} and  Proposition \ref{wholepoints} imply that  the  coordinates of  ${\bf O}_{X(\p)}$ given by rotation-normalized algorithm is well defined.
\end{proof}

Summarizing the  previous  results, we obtain the main result of this section.
\begin{thm}\label{parathm}   Let $\p=(\p_{1},\cdots, \p_{m})$ be an $m$-tuple of pairwise distinct
	positive points given by Proposition \ref{deg-case}. Then the $\pspn$-congruence class of $\p$
	is determined  uniquely by the  partition structure of $S(m)=\bigcup_{i=1}^k S_i$ and the  coordinates of  ${\bf O}_{X(\p)}$ given by rotation-normalized algorithm.
\end{thm}

Therefore the moduli space can be described as follows.
\begin{thm}\label{paramodthm} The moduli space of  $\p=(\p_{1},\cdots, \p_{m})$  given by Proposition \ref{deg-case} can be identified with the set $\bm_1\times \cdots \times \bm_k$, where $\bm_1\times \cdots \times \bm_k$  are the  coordinates of  ${\bf O}_{X(\p)}$ given by rotation-normalized algorithm.
\end{thm}

\section{Moduli space on $\bp(V_+)$ of case $m\geq 3$: regular cases}\label{sec7reg}
In this section, we describe the moduli space of configurations of quaternionic $(n-1)$-dimensional submanifolds  when  $V$  is not parabolic in conceptual style.  The basic idea  is  to find  a partition of  $S(m)=\{1,\cdots,m\}$ to perform rotation-normalized algorithm  in each  block.

We begin with 1-normalized matrix of $\p$.   Proposition \ref{regula-case} roughly shows that we can treat the mutually orthogonal blocks $\p_{S_i}=(\p_{s_{i1}},\cdots,\p_{s_{it_i}}),i=1,\cdots,s$  separately.  Equivalently, we can perform the rotation-normalized algorithm separately.  This is the structure of Gram matrix at  top level.
  For each block $\p_{S_i}$, there may still exist $0$s in $G(\p_{S_i})$.   We may need to  partition $S_i$  into more small blocks to perform rotation-normalized algorithm.
  We call such a partition process, together with similar $1$-normalized process in each small blocks,  {\it block-normalized algorithm}. The output of   block-normalized algorithm is a special kind of Gram matrix, which is still not unique and can be viewed as an equivalent class. We still need to apply rotation-normalized algorithm to get the parameters.

\medskip
We describe block-normalized algorithm conceptually as follows.

\medskip
  {\bf Block-normalized algorithm:}
	\begin{itemize}
	\item[Step 1:]
Let $O_{il}$ be the number of entries being zero in $il$th row  of   $G(\p_{S_i})$ and record  the set  of columns  of these entries being nonzero as $P_{il}$.  Let $n_i=\min\{O_{i1},\cdots, O_{it_i}\}$ and $K_i$ the set of indices $il$  such that  $O_{il}=n_i$.
Let  $c_{i1}$ be the smallest integer in $K_i$ and denote the corresponding $P_{il}$ of $c_{i1}$ as $S_{i1}$.  In other words,   $c_{i1}$ is the smallest index in $S_i=\{s_{i1},\cdots,s_{it_{i}}\}$  such that  the cardinality of nonzero entries in the $c_{i1}$th row of  $G(S_i)$ is the largest among those of  the others; the set of columns of nonzero entries is recorded  as $S_{i1}$. It is obvious that $c_{i1}\in S_{i1}$.

\item [Step 2:] Repeating the process in {\it Step 1} for the  remainder of $S_i-S_{i1}$, we obtain $c_{i2}$ and $S_{i2}$. It is obvious that we can continue this process only finite steps. We denote by  $\tau_i$   the  number of  steps  and record  the corresponding numbers in each step  as $c_{ij}$ and $S_{ij}$ for $1\le j \le \tau_i$.
Then we have
$$S_i=\bigcup_{j=1}^{\tau_i} S_{ij}.$$

\item [Step 3:]  In each subindex set  $S_{ij}$, we perform  the $c_{ij}$-normalized process  to $G(\p_{S_{ij}})$.  We denote such result of the sub Gram matrix as  $G_b(\p_{S_{ij}})$.   In other words, the entries of  $G_b(\p_{S_{ij}})$ have the following properties:
$$g_{tt} = 1,    g_{c_{ij} t}\geq 0,   g_{t c_{ij}}\geq 0,  t\in S_{ij}.$$
  As in (\ref{d-1}) of Section \ref{sec5v+}, we record the corresponding  normalized sub-diagonal matrix as $D_{ij}$. That is
$$G_b(\p_{S_{ij}})= G(\p_{S_{ij}}D_{ij}).$$
\item [Step 4:]    Let  \begin{equation}\label{d-regp} D_i={\rm diag}(D_{i1},\cdots, D_{i\tau_i}), D_b={\rm diag}(D_1,\cdots, D_s).\end{equation}
We define
\begin{equation}
G_b(\p_{S_i})= G(\p_{S_i} D_i)
\end{equation}
and
\begin{equation}G_b(\p)= G(\p D_b).\end{equation}
 \end{itemize}
\medskip

\begin{dfn}The  Gram matrix $G_b(\p)$  obtained by the above  block-normalized algorithm is called the  block-normalized matrix of  $G(\p)$.\end{dfn}

We mention that our strategy in block-normalized algorithm is from parts to entirety. We deal with the diagonal blocks separately.  In this scale $\p_{S_i} $ and $\p_{S_j} $ are totally  independent.   In each block  $\p_{S_i} $,  all processes   are explicitly recorded by  the corresponding   sub-diagonal matrices $D_i={\rm diag}(D_{i1},\cdots, D_{i\tau_i})$.  In this way the entries  $\langle \p_{s_{il}}, \p_{s_{id}}\rangle$  in  off-diagonal blocks of  $\p_{S_i} $ are all determined definitely by $D_i$. We describe the structure of $G_b(\p)$ in the following proposition in more details.

\begin{prop}\label{struct-subgram}
	 The block-normalized  Gram matrix $G_b(\p)$  has the following characteristics.
	\begin{itemize}
	\item[(1)]
If we view the block-normalized  Gram matrix $G_b(\p)$ in its permuted  position with index $S_i$, then $G_b(\p)$ consists of blocks  submatrix  $G_b(\p_{S_{i}})$, the entries of the  corresponding off-diagonal blocks  matrices are zero (see (\ref{partregu}) in Proposition  \ref{regula-case}).

	\item[(2)] In the $c_{i1}$th row (and column)  of submatrix  $G_b(\p_{S_{i}})$, the first  ${\rm Card}(S_{i1})$ entries  are nonzero real numbers, the others  are zeros (see Step 2 of block-normalized algorithm).

		\item[(3)] In the $c_{i2}$th row (and column) of submatrix  $G_b(\p_{S_{i}})$,  the entries with index between ${\rm Card}(S_{i1}) +1 $ and   ${\rm Card}(S_{i1}) +{\rm Card}(S_{i2})$  are nonzero real numbers,  the entries  with index bigger than   ${\rm Card}(S_{i1}) +{\rm Card}(S_{i2})$  are zeros; the entries in the $c_{ij}$th row (and column) of submatrix  $G_b(\p_{S_{i}})$ can be described similarly when $j=3,\cdots,\tau_i$.
	
			\item[(4)]   $G_b(\p_{S_{i}})$ can not be block diagonal according to our partition in Proposition \ref{regula-case}.
 \end{itemize}

	\end{prop}

Similarly to Lemma  \ref{lem-actsemi},  we have the following result.

\begin{lem}\label{actsemip}
Suppose that   $G(\q)$  is a block-normalized Gram matrix  $G_b(\p)$  for  $\p=(\p_{1},\cdots, \p_{m})$.   Then   $G(\q D_r)$  is  still a block-normalized Gram matrix  with  $$D_r={\rm diag}(\mu_1, \cdots, \mu_m)$$  if  only if  every $\mu_t$ with  $t\in S_{ij}$ is  the same quaternion of modulus $1$,i.e., $$\mu_t=\mu_{ij}, \forall t\in S_{ij}, \ \mbox{where} \  \mu_{ij}\in {\rm Sp}(1).$$
\end{lem}
Summarizing the previous treatments,  we have the following procedure.
\begin{thm}\label{noparathm}
	Let $\p=(\p_{1},\cdots, \p_{m})$ be an  $m$-tuple of pairwise distinct positive point given by Proposition \ref{regula-case}. We can assign the $\pspn$-congruence class of $\p$ a coordinate as follows.
	
	\begin{itemize}
	\item[(1)] Obtain a block-normalized  matrix $G(\p D_b)$  by performing  the block-normalized algorithm, where  $D_b$ is given by (\ref{d-regp}).

	\item[(2)]   Perform the rotation-normalized algorithm to each block $S_{ij}$ (as the case of $m$-tuple of $\bp(V_0)$ in  Section \ref{sec3mv0}).  This is equivalent to choosing a specific $\mu_{ij}\in {\rm Sp}(1)$.   Combine them to the corresponding whole rotation  normalized  diagonal  matrix $D_r$
	
	\item[(3)]  The independent entries of  $$G(\p D_bD_r),$$   that is, all the entries above the diagonal entries,  are the desired  coordinate of the $\pspn$-congruent class of $\p$.
 \end{itemize}
	\end{thm}

We now are ready to give a conceptual description of the parameter space $\bm(n,m)$ in regular case. We mimic  conceptually the method used in Section \ref{submpv0} as follows.

\medskip
{\bf  The procedure of  constructing parameter space:}

  For a partition $\mathcal{S}$ of $S(m)=\{1,\cdots,m\}$ as
  $$S_i=\{s_{i1},\cdots,s_{it_{i}}\},s_{i1}<\cdots<s_{it_{i}} ,i=1,\cdots,s$$  with sub partitions
  $$S_i=\bigcup_{j=1}^{\tau_i} S_{ij}.$$ Let ${\rm Card}(S_{ij})=\sigma_{ij}.$  As in  Section \ref{submpv0}, we construct the parameter  space $\bm(n,\sigma_{ij})$ of $S_{ij}$. Let $$\bm(n,i)=\bm(n,\sigma_{i1})\times \cdots \bm(n,\sigma_{i\tau_{i}})\times \mathbb{C}_i, $$
where the set of $\mathbb{C}_i$  is the corresponding  space of the off-diagonal sub-blocks.  Let
  $$\bm(n,m,\mathcal{S})= \bm(n,i)\times \cdots \bm(n,s).$$
 The Hermitian matrix constructed from the entries of the parameter space $\bm(n,m,\mathcal{S})$ should subject to analogous constraints as those of  Theorem \ref{thmcondgp}. Then the union of the parameter spaces determined by  all possible  partitions
\begin{equation}\bm(n,m)=\bigcup_{\mathcal{S}}\bm(n,m,\mathcal{S})\end{equation}
  is a parameter space of the configuration space $\mathcal{M}(n,m;\bp(V_+))$ when $V$  is not  parabolic.

\medskip

Therefore, the moduli space can be described as follows.
\begin{thm}\label{paramodthmr} The moduli space of  $\p=(\p_{1},\cdots, \p_{m})$  given by Proposition \ref{regula-case} can be identified with the set \begin{equation}\bm(n,m)=\bigcup_{\mathcal{S}}\bm(n,m,\mathcal{S}).\end{equation}
\end{thm}

\section{Quaternionic hyperbolic triangles}\label{sec8tri}
In this section, we will give  a parameter space of quaternionic hyperbolic triangles. This section may be regarded as an application of somewhat conceptual results in previous sections in   triangle groups, a current hot research topic since the seminal work of   Goldman and Parker \cite{goldpark92}.

 We will work on ball model and begin with some notations.  Let ${\p}_t$ be the normalized polar vector of the quaternionic line $l_t$, $t=1,2,3$.
That is  $l_t$ is  a quaternionic $1$-dimensional submanifold corresponding to ${\p}_t$ with $\langle {\p}_{t},{\p}_{t}\rangle=1$.
\begin{dfn}
A {\it quaternionic  hyperbolic triangle} is a triple $(l_1,l_2,l_3)$ of quaternionic lines in quaternionic hyperbolic space ${\rm H}_{\bh}^2$.
\end{dfn}
  For  pair of  quaternionic lines $l_{t-1}$ and $l_{t+1}$, let $r_t=|\langle {\p}_{t-1},{\p}_{t+1}\rangle|$, where the indices are taken mod 3.
 By Theorem \ref{thm-twpo},  the number $r_t<1$, $r_t=1$ and $r_t>1$ means that  the quaternionic lines $l_{k-1}$ and $l_{k+1}$ intersect at ${\rm H}_{\bh}^2$ with  angle $\varphi_t=\arccos r_t$,  intersect at $\partial {\rm H}_{\bh}^2$, are ultra-parallel with the distance $\ell_t=2\cosh^{-1} r_t$, respectively.

We define that the following quaternion for a triple of points  ${\p}= ({\p}_{1},{\p}_{2},{\p}_{3})$ in $V_+$: $$\langle {\p}_{1},{\p}_{2},{\p}_{3} \rangle=:\langle {\p}_{2},{\p}_{1}\rangle \langle {\p}_{3},{\p}_{2}\rangle \langle {\p}_{1},{\p}_{3}\rangle. $$

 \begin{dfn}\label{auglar}
The angular invariant of the  quaternionic  hyperbolic triangle  $(l_1,l_2,l_3)$ is defined by
$$\ba({\p})=\ba({\p}_{1},{\p}_{2},{\p}_{3})=\left\{
               \begin{array}{ll}
                 \arccos \frac{\Re( \langle {\p}_{1},{\p}_{2},{\p}_{3} \rangle)}{| \langle {\p}_{1},{\p}_{2},{\p}_{3} \rangle|}, & \hbox{provided}\   \langle {\p}_{1},{\p}_{2},{\p}_{3} \rangle\neq 0;\\
                 \pi/2, & \hbox{otherwise.}
               \end{array}
             \right.$$
 \end{dfn}
 When $\ba({\p})=\pi/2$, we may have two cases:  $\langle {\p}_{1},{\p}_{2},{\p}_{3} \rangle=0$ or $\Re(\langle {\p}_{1},{\p}_{2},{\p}_{3} \rangle)=0$. For example, $\langle {\p}_{1},{\p}_{2},{\p}_{3} \rangle={\bf i}$ when  $\p_1=(0,1,0)^T, \p_2=(1,1,1)^T, \p_3=({\bf i},1,1)^T$.

 It is obvious  that  $$\ba({\p})=\ba({\p}_{1},{\p}_{2},{\p}_{3})= \ba(f{\p}_{1},f{\p}_{2},f{\p}_{3})\in [0,\pi], \forall f\in {\rm Sp}(2,1). $$
 It is easy to verify the following proposition.
 \begin{prop} Let ${\p}_{1},{\p}_{2},{\p}_{3}\in V_+$, let $\sigma$ be  a permutation of $1,2,3$, let $\lambda_i\in \bh- \{0\}$. Then
$$\ba({\p}_{1},{\p}_{2},{\p}_{3})=\ba({\p}_{\sigma(1)},{\p}_{\sigma(2)},{\p}_{\sigma(3)})=\ba({\p}_{1}\lambda_1,{\p}_{2}\lambda_2,{\p}_{3}\lambda_3).$$
 \end{prop}
In this specific case, as in the $1$-normalized process, we have the following proposition.
\begin{prop}\label{normgram}
Let ${\p}= ({\p}_{1},{\p}_{2},{\p}_{3})$ be a triple of points in  $V_+$.  Then the equivalence class of Gram matrices associated to ${\p}$  contains a
unique matrix $G= (g_{ij})$ with $g_{ii} = 1$, $g_{12}=r_1\geq 0$,  $g_{13}=r_2\geq 0$ and $g_{23}=r_1(\cos\ba+\sin\ba{\bf i}) \in \bc$, where $\sin\ba\ge 0$.
\end{prop}
\begin{proof} By appropriate rescaling we may assume that ${\p}_{i}$ are normalized vectors, i.e., $g_{ii} = 1, i=1,2,3$.
For $i=2,3$, let
 $$\lambda_i=\left\{
               \begin{array}{ll}
                 \frac{\langle {\p}_{1},{\p}_{i} \rangle}{|\langle {\p}_{1},{\p}_{i}\rangle|}, & \hbox{provided}\   \langle{\p}_{1},{\p}_{i}\rangle\neq 0;\\
                 1, & \hbox{otherwise.}
               \end{array}
             \right.$$
It is known that there exists a $\lambda_1$ of norm 1 such that  $\bar{\lambda_1}\bar{\lambda_3}\langle {\p}_{2},{\p}_{3}\rangle \lambda_2\lambda_1$ is a complex number with no-negative imaginary part.  Then  $G({\p}_{1}\lambda_1,{\p}_{2}\lambda_2\lambda_1,{\p}_{3}\lambda_3\lambda_1)$ is the desired Gram matrix.
\end{proof}

The Gram  matrix in Proposition \ref{normgram} is called the normalized Gram matrix, which is of the form
 \begin{equation}\label{normmatrix}G=(g_{ij})=\left(
                  \begin{array}{ccc}
                    1 & r_3 & r_2 \\
                    r_3 & 1 & r_1e^{{\bf i}\ba} \\
                    r_2 & r_1e^{-{\bf i}\ba} & 1\\
                  \end{array}
                \right).\end{equation}
We call such a quaternionic hyperbolic  triangle a  $\left(r_1,r_2,r_3;\ba\right)$-triangle.

Let $G(i,j)$ be the submatrix consisting of entries in row and column index with $i,j$.  It is easy to verify the following proposition.
\begin{prop}
Let $G$ be the normalized matrix of  $(r_1,r_2,r_3;\ba)$ triangle.
	\begin{itemize}
	\item[(1)] Since $i(G)\neq (3,0,0)$, $r_1^2+r_2^2+r_3^2\neq 0$.
\item[(2)]
\begin{equation}
\det G(1,2)=1-r_3^2, \ \det G(1,3)=1-r_2^2,\  \det G(2,3)=1-r_1^2
\end{equation}
and
\begin{equation}
\det G=1-(r_1^2+r_2^2+r_3^2)+2r_1r_2r_3\cos\ba.
\end{equation}
	\end{itemize}
\end{prop}

Let $V={\rm span}\{\p_1,\p_2, \p_3\}$. We assume that $\p_1,\p_2, \p_3$ are pairwise distinct points in $\bp(V_+)$, therefore $\dim_q(V) \ge 2$.  We mention that beginning  with three points in $\overline{{\rm H}_{\bh}^2}$ and then constructing the quaternionic lines,
 one may obtain that the corresponding polar vectors  $\p_1,\p_2, \p_3$  which may be the same  in the view point of $\bp(V_+)$.
 In this situation the three points lie in the closure of a common quaternionic line and  $\dim_q(V)=1$,
 however this case is not so interesting \cite{caohuang16}.
\begin{prop}\label{tricase}
 We enumerate the possibilities of the signatures  $i(G)$  corresponding to  quaternionic  hyperbolic triangle groups.
\begin{itemize}
	\item[(1)] If  $V$ is parabolic then $i(G)=(1,0,0)$,  $\dim_q(V)=2$ and it corresponds to $(1,1,1;0)$-triangle.
\item[(2)]If  $V$ is elliptic then $i(G)=(2,0,0), \dim_q(V)=2$.
\item[(3)]If $V$ is hyperbolic then $i(G)=(1,1,0), \dim_q(V)=2$ or $(2,1,0), \dim_q(V)=3$.
\end{itemize}
\end{prop}
It follows from Proposition \ref{threep} that the  parameter space of $(1,1,1;0)$-triangle is $\bc-\{0,1\}$.
\medskip

By Proposition \ref{tricase} and the properties of determinant of complex matrices, we have the following result.
\begin{thm}\label{tri-para}
For any $\ba\in[0,\pi/2]$  there exists a  quaternionic  hyperbolic $(r_1,r_2,r_3;\ba)$-triangle  in ${\rm H}_{\bh}^2$  if and only if
\begin{equation}\label{condt}\det G=1-(r_1^2+r_2^2+r_3^2)+2r_1r_2r_3\cos\ba\leq 0.\end{equation}
Moreover,  $\det G=0$  if and only if  there exist  $\lambda_i,\in \bh$ with $\sum_{i=1}^3|\lambda_i|\neq 0$  such that ${\p}_{1}\lambda_1+{\p}_{2}\lambda_2+{\p}_{3}\lambda_3=0$.
\end{thm}

We need to replace $\ba\in[0,\pi/2]$ with $\ba\in[-\pi,\pi]$  in  Theorem \ref{tri-para}  for  complex hyperbolic geometry.  We refer to \cite{pra05,sch02} etc.  for more details of the complex hyperbolic triangle groups. Cao and Huang \cite{caohuang16} have  addressed the discreteness of quaternionic ideal triangle groups, which corresponds to the quaternionic  hyperbolic triangle of the types $(1,1,1;\ba)$. We mention that the angular invariant $\ba$ given by   Definition \ref{auglar} is different from that of \cite{caohuang16}.
It is of  current interest to settle  the  problems of  faithful and discrete presentations both on complex and quaternionic hyperbolic geometries.

\vspace{3mm}

{\bf Acknowledgements}\quad  I am grateful to  John R. Parker and Ioannis D. Platis  for suggesting the problem to me, as well as many useful comments.   This work  was  supported by State  Scholarship  Council of China and  NSF of Guangdong Province (2015A030313644) and  completed when the author was an Academic Visitor in Durham University. He would like to thank  Department of Mathematical Sciences  for its hospitality.

\end{document}